\documentclass[10pt,A4page]{amsart}
\usepackage{amssymb}

\theoremstyle{plain}
\newtheorem{prop}{Proposition}[section]
\newtheorem{thm}[prop]{Theorem}
\newtheorem{coro}[prop]{Corollary}

\theoremstyle{definition}

\newtheorem{exam}[prop]{Example}
\newtheorem{rem}[prop]{Remark}
\newtheorem{question}[prop]{Question}

\newtheorem{problem}[prop]{Problem}

\theoremstyle{remark}

\numberwithin{equation}{section}

\newcommand{\N}{\mathbb N}
\newcommand{\Z}{\mathbb Z}
\newcommand{\Q}{\mathbb Q}
\newcommand{\R}{\mathbb R}
\newcommand{\C}{\mathbb C}

\newcommand{\G}{\Gamma}

\newcommand{\ba}{\backslash}

\newcommand{\Ot}{\operatorname{O}}
\newcommand{\SO}{\operatorname{SO}}

\newcommand{\Hom}{\operatorname{Hom}}

\newcommand{\so}{\mathfrak{so}}
\newcommand{\op}{\operatorname}

\newcommand{\diag}{\operatorname{diag}}

\newcommand{\triv}{\mathbf{1}}

\newcommand{\norma}[1]{\|{#1}\|_1}

\newcommand{\ww}{\theta}
\newcommand{\zz}{\ell}

\newcommand{\QQ}{Q}

\title[Non-strongly spherical space forms]{Non-strongly isospectral spherical space forms}
\author{E. A. Lauret, R. J. Miatello and J. P. Rossetti}
\address{CIEM--FaMAF \\ Universidad Nacional de C\'ordoba\\ 5000-C\'ordoba, Argentina.}
\email{elauret@famaf.unc.edu.ar}
\email{miatello@famaf.unc.edu.ar}
\email{rossetti@famaf.unc.edu.ar}
\date{March 25, 2015}

\begin{document}

\maketitle

\begin{abstract}
In this paper we describe recent results on   explicit construction of lens spaces that are not strongly isospectral, yet they are isospectral on $p$-forms for every $p$. Such examples cannot be obtained by the Sunada method. We also discuss related results, emphasizing on significant classical work of Ikeda on isospectral lens spaces, via a thorough study of the associated generating functions.
\end{abstract}

\section{Introduction}

Two compact Riemannian manifolds are said to be isospectral if the spectra of their Laplace operators on functions are the same.
More generally, they are said to be $p$-isospectral if the spectra of  their Hodge-Laplace operators acting on $p$-forms are the same.

Recently, in \cite{LMRhodge}, we have found examples of pairs of lens spaces that are $p$-isospectral for every $p$.
Since lens spaces have cyclic fundamental group, they cannot be \emph{strongly isospectral}.
To the best of our knowledge these are the first (connected) examples of this kind.
By showing a nice connection between isospectrality of lens spaces and isospectrality of certain
associated integral lattices with respect to the one-norm, we were able to
construct an infinite family of pairs of $5$-dimensional lens spaces that are $p$-isospectral for every $p$.

Before this,
A.~Ikeda found many interesting examples of isospectral lens spaces. The main tool of his approach was the generating function associated to the spectrum. Our method does not use generating functions, but relies on the representation theory of compact Lie groups.

In view of our construction of new families and the opening connection with one-norm isospectral integral lattices, we expect it will be useful to write this article attempting to bring together in a more accessible way, our method, the foundational work of Ikeda and the method of Sunada.

Historically, the first example of isospectral non-isometric manifolds was a pair of tori constructed by using lattices of dimension $n=16$ (\cite{Mi64}, \cite{Wi}). The dimension was reduced from 16 to 4 in several articles (see \cite{Schi1}, \cite{CS2} and the references therein).
Such lattices are isospectral with respect to the standard norm $\|\cdot\|_2$, that is, for each length they have the same number of vectors of that length.

Besides these examples, many other contributions
have been given, showing different connections
between the spectra and the geometry of a Riemannian manifold.
In \cite{Su} T.~Sunada  gave a general method  that  produces strongly isospectral manifolds,
that is, manifolds isospectral for every natural strongly elliptic operator acting on sections of a natural
vector bundle, in particular they are $p$-isospectral for all $p$. Later, this method was extended and applied by many authors, in particular by D.~DeTurck-C.~Gordon \cite{DG89}, P.~B\'erard \cite{Be93} and H.~Pesce \cite{Pe96}.
In the context of spherical space forms, A.~Ikeda \cite{Ik83},
P.~Gilkey~\cite{Gi}, J.~A.~Wolf~\cite{Wo2}  produced  Sunada isospectral
forms with non-cyclic fundamental groups.

 The construction of manifolds that are $p$-isospectral for some values of $p$ only cannot be attained by Sunada's method.
 The first such pair was given by C.~Gordon in \cite{Go}.
 Among other known examples we mention those in \cite{Gt00} for nilmanifolds and those  given in \cite{MRp}, \cite{MRl}, \cite{DR1} for compact flat manifolds.

A.~Ikeda studied the spectrum of spherical space forms in several interesting articles (see \cite{Ik80a}, \cite{Ik80b},\cite{Ik80c}, \cite{Ik83}, \cite{Ik88}). He developed the theory of generating functions associated to spectra, obtaining many isospectral examples of Sunada and non-Sunada type.
In particular, for each given $p_0$, he constructed families of lens spaces that are $p$-isospectral for every $0\le p\le p_0$, but are not $p_0+1$-isospectral.
None of Ikeda's examples of isospectral lens spaces are $p$-isospectral for all $p$
and actually until very recently, no examples were known of compact Riemannian manifolds
that are $p$-isospectral for every $p$ but are not strongly isospectral.
This question has been around for some time (see \cite[p.~323]{Wo2}).
In  \cite{LMRhodge} we find a rather surprising two-parameter infinite family of pairs of lens spaces that are $p$-isospectral  for every $p$,  but are not strongly isospectral. We also give many more examples obtained with the help of the computer and also examples in arbitrarily large dimensions.

The paper is organized as follows. Section~2 is devoted to describe summarily Sunada's method and its generalizations. In Section~3 we develop the necessary tools of representation theory of compact Lie groups to be used in the proofs of our main results in Section~\S 5. Section~4 is devoted to  Ikeda's important work, that is scattered in several papers that are  sometimes hard to follow. We have tried to make it more accessible, including the main ideas in most of the proofs. In Section 5 we describe our construction of isospectral lens spaces in dimension $n=2m-1$ by means of one-norm isospectral integral lattices in $\Z^m$. A detailed description of the methods and the results is given at the beginning of the section. The paper finishes with tables, obtained by computer methods, listing all existing examples for $n=5$, 7 and 9,  where the order of the fundamental group $q$ is less than 500, 300 and 150 respectively.
We have left some open questions or problems, usually at the end of the sections or subsections.

\noindent {\it Acknowledgement.} The authors wish to thank Peter Doyle for stimulating discussions and for facilitating the use of fast computer programs to check the tables in Section~5.

\section{Sunada's method}
T.~Sunada~\cite{Su} gave a simple and effective method that allowed to produce a great variety of examples of isospectral manifolds.
It is based on a triple of finite groups $\G_1,\G_2, G$, where $\G_1,\G_2$ are subgroups of $G$ that are almost conjugate in $G$, that is, there is a bijection from $\G_1$ to $\G_2$ that preserves  $G$-conjugacy.
The first such triples were given by Gassmann~\cite{Ga26} who used them to give pairs of non-isomorphic number fields having the same Dedekind zeta function.
The Sunada theorem can be stated as follows.

\begin{thm}\label{thm:Sun}
Let $\G_1,\G_2$ be  almost conjugate subgroups of a finite group $G$.
Assume that $G$ acts by isometries on a Riemannian manifold $M$ in such a way that $\G_1,\G_2$ act freely.
Then the manifolds $\G_1\backslash M$ and $\G_2\backslash M$ are strongly isospectral.
\end{thm}

Given a Gassmann triple $\G_1,\G_2, G$, to place oneself in the conditions of Sunada's theorem it is sufficient to give a Riemannian manifold $M_0$ such that there is a surjective homomorphism $\phi: \pi_1(M_0)\rightarrow G$.
Sunada gave many applications of this theorem, in particular, he constructed large sets of pairwise isospectral non-isometric Riemann surfaces for any genus $g\ge 5$. Also, he showed that manifolds $\G_1\backslash M$ and $\G_2\backslash M$ as in the theorem must have the same lengths of closed geodesics. We note, however, that these lengths need not have the same multiplicities (see for instance \cite{Go85}, \cite{Gt94}, \cite{Gt96} and \cite{MRl}).

Sunada's result was intensely exploited and was followed by several generalizations. Still today, the method accounts for most of the known examples of isospectral manifolds.  We note that the condition of almost conjugacy in the finite group $G$ is equivalent to a condition in terms of group representations, namely, that the right
regular representations of $G$ on the function spaces $C(\G_1\backslash G)$ and $C(\G_2\backslash G)$ are equivalent representations.
More generally, if $G$ is a Lie group and $\Gamma_1, \Gamma_2$ are discrete cocompact subgroups, then $\G_1$, $\G_2$ are said to be \emph{representation equivalent} in $G$, if the right
regular representations of $G$ on $L^2(\G_1\backslash G)$ and $L^2(\G_2\backslash G)$ are equivalent representations.

The following generalization of Theorem~\ref{thm:Sun}, due to DeTurck-Gordon~\cite{DG89} (see also \cite{Be93}), is very useful.

\begin{thm}\label{thm:dTG}
Let $G$ be a Lie group acting by isometries on a Riemannian manifold $M$ and let $\G_1,\G_2$ be  discrete subgroups of $G$ such that  $\G_1\backslash M$ and $\G_2\backslash M$ are compact manifolds.
If, furthermore, $\G_1$, $\G_2$ are representation equivalent in $G$, then $\G_1\backslash M$ and $\G_2\backslash M$ are strongly isospectral.
\end{thm}

One can give a convenient reformulation of the condition of representation equivalence in the theorem. Namely,  if $g \in G$ denote by $C(g, \G_i)$, $C(g, G)$  the centralizers of $g$ in $\G_i$ and $G$ respectively.
Under the conditions above, the quotient
$C(g, \G_i)\backslash C(g, G)$ is compact for $i=1,2$.
One has that $\G_1$, $\G_2$ are representation equivalent in $G$ if and only if, for each $g \in G$,
\begin{equation}
\sum _{{[a]_{\G_1}}\subset {[g]_{G}}} \textrm{vol} (C(a, \G_1)\backslash C(a, G))
= \sum _{{[b]_{\G_2}}\subset {[g]_{G}}} \textrm{vol} (C(b, \G_2)\backslash C(b, G)).
\end{equation}
Here $[g]_G$, $[a]_{\G_1}$ and $[b]_{\G_2}$ denote respectively the conjugacy classes of $g$ in $G$, of $a$ in $\G_1$ and of $b$ in $\G_2$ and the volumes are computed with respect to suitable invariant measures in  $C(g, G)$ and $C(g, \G_i)$.

As a consequence, one obtains the following (see \cite{Wo2})

\begin{coro}
If $G$ is a compact Lie group and the $\G_i$ are finite, then $\G_1$ and $\G_2$ are representation equivalent in $G$ if and only if  $\G_1$, $\G_2$ are almost conjugate in $G$.
\end{coro}

H.~Pesce studied the relation between representation theory and isospectrality in several papers (\cite{Pe95}, \cite{Pe96}, \cite{Pe98}). In particular, in \cite{Pe95} he proved that the converse of the Sunada condition is satisfied for manifolds of curvature $\pm 1$.
That is, if $X=S^n$ or $X=H^n$, $G=I(X)$ and $\G_1$ and $\G_2$ are discrete cocompact subgroups of $G$ and if $\G_1\ba X$ and $\G_2\ba X$ are strongly isospectral, then $\G_1$ and $\G_2$ are representation equivalent in $G$. (In the case of $X=\R^n$ this result is proved in \cite{L}.)

Also, in a subsequent paper (\cite{Pe96}), he gives a generalization of the condition of representation equivalence, by introducing the weaker notion of \emph{$\tau$-representation equivalence}, where $\tau$ is a representation of the compact Lie group $K$ that is the generic stabilizer of the action.
When $\tau$ is the trivial representation he calls the discrete subgroups $K$-equivalent. As an application, in the case of spaces of constant curvature he shows that $\G_1$ and $\G_2$ are $K$-equivalent in $G=I(X)$ if and only if $\G_1\ba X$ and $\G_2\ba X$ are $0$-isospectral.
A generalization  was given recently in \cite{LMRrepequiv}, in the same context, for $X=G/K$ of constant curvature  when $\tau=\tau_p$ is the $p$-exterior representation of $K=\Ot(n)$.
We showed that in the elliptic case, for each fixed $p$, $\G_1$ and $\G_2$ are  $\tau_p$-equivalent in $G=I(X)$ if and only if $\G_1\ba X$ and $\G_2\ba X$ are $p$-isospectral.
However, in the flat and hyperbolic cases, we prove that $\G_1$ and $\G_2$ are $\tau_q$ equivalent in $G=I(X)$ for every $0\le q \le p$, if and only if $\G_1\ba X$ and $\G_2\ba X$ are  $q$-isospectral for every  $0\le q \le p$.
Also we gave examples showing that  in the flat case,  $p$-isospectrality is far from implying $\tau_p$-equivalence for each fixed $p$.

To conclude this section, we list a number of representative papers illustrating the construction of strongly isospectral manifolds by means of the Sunada method or its generalizations.
\begin{enumerate}
\item [(i)] Isospectral Riemann surfaces: \cite{Vi}, \cite{Su}, \cite{BT}, \cite{BGG}, \cite{Br96}, \cite{BuLibro}, \cite{Bu86}, \cite{GMW}.
\item [(ii)] Isospectral spherical space forms: \cite{Ik83}, \cite{Gi},  \cite{Wo2}.
\item [(iii)] Isospectral locally symmetric: \cite{Vi}, \cite{Sp}, \cite{McR}.
\item [(iv)] Continuous isospectral families: \cite{GW84}, \cite{GW97}, \cite{Schu95}.
\item [(v)] Isospectral graphs: \cite{Bu88}, \cite{Br96}, \cite{FK}.
\item [(vi)] Isospectral planar domains: \cite{GWW}, \cite{Bu88}, \cite{BCDS}.
\item [(vii)] Isospectral flat manifolds: \cite{DM}, \cite{DR1}, \cite{MRhw}, \cite{MRl}, \cite{LMRcohom}.
\item [(viii)] $\tau$-representation equivalent manifolds: \cite{Pe96}, \cite{Pe98},  \cite{Su02}, \cite{LMRrepequiv}.
\end{enumerate}

For a more complete discussion of the Sunada method we refer to the surveys by C.~Gordon~\cite{Go08}, \cite{Go00}.

In the remaining sections we will discuss several isospectrality situations in the case of spherical space forms, that are not of the strong type, thus they cannot be obtained by the Sunada method.

\section{Spectra of spherical space forms} \label{sec:specspherforms}

In this section we will recall various facts on spectra of spherical space forms. We refer to \cite{IT} for the main basic facts.
We will describe the results  in the language of representation theory of orthogonal groups.
The $n$-dimensional sphere $S^{n}$ is a symmetric space realized as $G/K$ with $G=\SO(n+1)$, $K=\SO(n)$.
If  $\Gamma$ is a finite subset of $\SO(n+1)$ acting freely on $S^n$, then the manifold $\Gamma\ba S^{n}$ is a spherical space form.
We restrict our attention to the odd-dimensional case $n=2m-1$, since the only manifold covered (properly) by $S^{2m}$ is $P\R^{2m}$.

We consider the standard maximal torus $T$ in $\SO(2m)$, with Lie algebra given by
\begin{equation}\label{eq2:h_0}
\mathfrak h_0:=\left\{ H=
\diag\left(
\left[\begin{smallmatrix}0&2\pi \theta_1\\ -2\pi \theta_1&0\end{smallmatrix}\right]
, \dots,
\left[\begin{smallmatrix}0&2\pi \theta_m\\ -2\pi \theta_m&0\end{smallmatrix}\right]
\right)
:\theta\in\R^m
\right\}.
\end{equation}
Its complexification is a Cartan subalgebra
$\mathfrak h$ of $\so(2m,\C)$.
As usual, define $\varepsilon_j\in\mathfrak h^*$ by $\varepsilon_j(H)=-2\pi i\theta_j$ for any $1\leq j\leq m$, $H\in \mathfrak h$.
The weight lattice of $G$ is thus given by $P(\SO(2m))=\bigoplus_{j=1}^m\Z\varepsilon_j$.
We use the standard system of positive roots, thus a weight $\sum_{j=1}^m a_j\varepsilon_j$ is dominant if and only if $a_1\geq\dots\geq a_{m-1}\geq |a_m|$.

Let $\langle\cdot,\cdot\rangle$ be the inner product on $i\mathfrak h_0^*$ so that $\varepsilon_1,\dots,\varepsilon_m$ is an orthonormal basis.
This is the dual of the positive multiple of the Killing form that induces on $S^{2m-1}$ the Riemannian metric with constant curvature equal to one.

In $K=\{g \in \SO(2m): g e_{2m} = e_{2m}\}\subset \SO(2m)$, we take the maximal torus $T_K = T\cap K$, thus the associated Cartan subalgebra $\mathfrak h_K$ can be seen as included in $\mathfrak h$ in the usual way.
Under this convention, the weight lattice of $K$ can be identified with $P(\SO(2m-1))=\bigoplus_{j=1}^{m-1} \Z \varepsilon_j$ and $\sum_{j=1}^{m-1}a_j\varepsilon_j$ is dominant if and only if $a_1\geq\dots \geq a_m\geq0$.

Let $\widehat G$ and $\widehat K$ denote respectively the equivalence classes of unitary irreducible representations of $G$ and $K$ respectively, endowed with invariant inner products.
By the highest weight theorem, the elements in $\widehat G$ (resp.\ $\widehat K$) are in a bijective correspondence with the dominant weights $\Lambda$ of $G$ (resp. $\mu$ of  $K$).
For each $\Lambda$, we denote by $\pi_\Lambda \in \widehat G$  the irreducible representation with highest weight $\Lambda$.
For example, $\pi_{k\varepsilon_1}\in\widehat G$, with highest weight $k\varepsilon_1$, can be realized in the space of complex homogeneous harmonic polynomials of degree $k$, in $m$ variables.

For $(\tau, W_\tau)\in\widehat K$, let $E_\tau$ denote the associated homogeneous vector bundle $E_\tau := G \times_\tau W_\tau \longrightarrow S^{2m-1}$  of $S^{2m-1}$ (see \cite[\S5.2]{Wa}).
The space of $L^2$-sections of $E_\tau$ decomposes as $L^2(E_\tau) \simeq \sum_{\pi\in\widehat G} \; V_\pi\otimes\Hom_K(V_\pi,W_\tau)$, where  $G$ acts in the first variable in the right-hand side.
If $\Gamma$ is a finite subgroup of $G$, the space $\Gamma\ba E_\tau$ is a vector bundle over $\Gamma\ba S^{2m-1}$ with $L^2$-sections given by the $\Gamma$-invariant elements of $L^2(E_\tau)$; thus we have the decomposition
\begin{equation}\label{eq:L^2(GammaE_tau)}
L^2(\Gamma\ba E_\tau) = L^2(E_\tau)^\Gamma \simeq \sum_{\pi\in\widehat G} \; V_\pi^\Gamma\otimes\Hom_K(V_\pi,W_\tau).
\end{equation}

The Laplace operator $\Delta_{\tau,\Gamma}$ acting on smooth sections $\Gamma\ba E_\tau$ can be identified with the action of the Casimir element $C\in U(\so(2m,\C))$ (the universal enveloping algebra of $\so(2m,\C)$).
On each summand $V_\pi^\Gamma\otimes\Hom_K(V_\pi,W_\tau)$, $C$ acts  by the scalar $\lambda(C,\pi) = \langle \Lambda+\rho,\Lambda+\rho\rangle -\langle \rho,\rho\rangle$, where $\Lambda$ is the highest weight of $\pi$ and $\rho = \sum_{j=1}^m (m-j)\varepsilon_j$.
In particular, the multiplicity $d_\lambda(\tau,\Gamma)$ of $\lambda \in\R$ in the spectrum of $\Delta_{\tau,\Gamma}$ equals
\begin{equation}\label{eq:mult_lambda}
d_\lambda(\tau,\Gamma)=
\sum_{\pi\in\widehat G:\, \lambda(C,\pi)=\lambda}
\dim V_\pi^\Gamma \;[\tau:\pi],
\end{equation}
where $[\tau:\pi]=\dim(\Hom_K(V_\pi,W_\tau))$ can be computed by the well known branching law from $G=\SO(2m)$ to $K=\SO(2m-1)$.
That is, if $\tau\in\widehat K$ has highest weight $\mu=\sum_{j=1}^{m-1} b_j\varepsilon_j$ and $\pi\in\widehat G$ has highest weight $\Lambda=\sum_{j=1}^m a_j\varepsilon_j$, then $[\tau:\pi]>0$ if and only if
\begin{equation}\label{eq:entrelazamiento}
a_1\geq b_1\geq a_2\geq b_2\geq\dots\geq a_{m-1}\geq b_{m-1}\geq |a_m|.
\end{equation}
Moreover, the branching is multiplicity free.
Hence, $\widehat G_\tau:=\{\pi\in \widehat G: [\tau:\pi]=1\}$ is the set of $\pi_\Lambda\in\widehat G$ ($\Lambda=\sum_{j=1}^ma_j\varepsilon_j$) such that \eqref{eq:entrelazamiento} holds.

We can now describe the $\tau$-spectrum of any spherical space form $\Gamma\ba S^{2m-1}$.

\begin{thm}\label{thm:tau-spectrum}
Let $\Gamma$ be a finite subgroup of $G=\SO(2m)$ and let $\tau$ be an irreducible representation of $K=\SO(2m-1)$.
Then, $\lambda\in\R$ is an eigenvalue of $\Delta_{\tau,\Gamma}$ if and only if $\lambda=\lambda(C,\pi)$ for some $\pi\in\widehat G_\tau$.
In this case, its multiplicity is given by
$$
d_{\lambda}(\tau,\Gamma) = \sum
    \dim V_{\pi}^\Gamma,
$$
the sum taken over $\pi\in\widehat G_\tau$ such that $\lambda(C,\pi)=\lambda$.
\end{thm}

If $\tau$ is an irreducible representation of $K=\SO(2m-1)$, then the spaces $\Gamma\ba S^{2m-1}$ and $\Gamma'\ba S^{2m-1}$ are said to be \emph{$\tau$-isospectral} if the Laplace type operators $\Delta_{\tau,\Gamma}$ and $\Delta_{\tau,\Gamma'}$ have the same spectrum.
If $\tau_p$ denotes the irreducible representation of $\SO(2m-1)$ with highest weight $\varepsilon_1+\dots+\varepsilon_p$ for $0\leq p\leq m-1$, then the associated Laplace operator $\Delta_{\tau_p,\Gamma}$ can be identified with the Hodge-Laplace operator $\Delta_p$ acting on $p$-forms of $\Gamma\ba S^{2m-1}$.
As usual, we call \emph{$p$-spectrum} the spectrum of $\Delta_p$ and we write \emph{$p$-isospectral} in place of $\tau_p$-isospectral.
Since $\Gamma\subset \SO(2m)$, then $\Gamma\ba S^{2m-1}$ is always orientable, hence  the $p$-spectrum and the ${2m-1-p}$-spectrum are the same.

We next restate Theorem~\ref{thm:tau-spectrum} for $\tau=\tau_p$.
We first introduce some more notation.
Let $\Lambda_p=\varepsilon_1+\dots+\varepsilon_p$ for $p<m$ and $\Lambda_m^\pm = \varepsilon_1+\dots+\varepsilon_{m-1}\pm\varepsilon_m$.
Denote by $\pi_{k,p}$ (resp.\ $\pi_{k,m}^\pm$) the irreducible representation with highest weight $k\varepsilon_1+\Lambda_p$ if $p<m$ (resp.\ $k\varepsilon_1+\Lambda_m^\pm$).
It is easy to check that
\begin{equation}
\widehat G_{\tau_p} =
\begin{cases}
\{\triv \}\cup\{\pi_{k,1}:k\in\N_0\}= \{\pi_{k\varepsilon_1}:k\in\N_0\}
    &\text{ if } p=0,\\
\{\pi_{{k,p}}, \pi_{{k,p+1}} : k\in\N_0\}
    &\text{ if } 1\leq p\leq m-2,\\
\{\pi_{{k,m-1}}, \pi_{{k,m}}^\pm : k\in\N_0\}
    &\text{ if } p=m-1.
\end{cases}
\end{equation}
Here $\triv$ denotes the trivial representation $\pi_{0}$ of $\SO(2m)$.
We now set $\mathcal E_0 =\{0\}$ and
\begin{equation}\label{eq:Ep}
\mathcal E_p = \{\lambda_k:=\lambda(C,\pi_{{k,p}})= k^2+k(2m-2)+(p-1)(2m-1-p): k\in\N_0\}
\end{equation}
for $1\leq p\leq m$.

\begin{thm}\label{thm:p-spectrum}
Let $\Gamma$ be a finite subgroup of $G=\SO(2m)$ and let $0\leq p\leq m-1$.
If $\lambda\in\R$ is an eigenvalue of $\Delta_{\tau_p,\Gamma}$ then $\lambda\in\mathcal E_p\cup\mathcal E_{p+1}$.
Its multiplicity is given by
$$
d_\lambda(\tau_p,\Gamma) =
\begin{cases}
\dim V_{\pi_{{k,p}}}^\Gamma & \text{if }\lambda=\lambda_k\in\mathcal E_p, \\
\dim V_{\pi_{{k,p+1}}}^\Gamma & \text{if }\lambda=\lambda_k\in\mathcal E_{p+1} .
\end{cases}
$$
In particular, when $p=0$, the eigenvalues of the Laplace-Beltrami operator $\Delta_{\tau_0,\Gamma}$ lie in the set $\{k^2+k(2m-2): k\in\N_0\}$ and $d_\lambda(\tau_0,\Gamma) = \dim V_{\pi_{k\varepsilon_1}}^\Gamma$ if $\lambda=k^2+k(2m-2)$.
\end{thm}

From Theorem~\ref{thm:p-spectrum} and the fact that $\mathcal E_p\cap\mathcal E_{p+1}=\emptyset$ when $p>0$, we obtain the following characterizations.

\begin{coro}\label{cor:isosp-dimV^Gamma}
Let $\Gamma$ and $\Gamma'$ be finite subgroups of $\SO(2m)$.
\begin{enumerate}
  \item[(i)] $\Gamma\ba S^{2m-1}$ and $\Gamma'\ba S^{2m-1}$ are $0$-isospectral if and only if $\dim V_{\pi_{k\varepsilon_1}}^\Gamma\!\! = \dim V_{\pi_{k\varepsilon_1}}^{\Gamma'}$ for every $k\in\N$.
  \item[(ii)] If $1\le p\le m-1$, $\Gamma\ba S^{2m-1}$ and $\Gamma'\ba S^{2m-1}$ are $p$-isospectral if and only if $$\dim V_{\pi_{{k,p}}}^\Gamma\! = \dim V_{\pi_{{k,p}}}^{\Gamma'} \textrm{ and } \dim V_{\pi_{{k,p+1}}}^\Gamma = \dim V_{\pi_{{k,p+1}}}^{\Gamma'}$$ for every $k\in\N$.
  \item[(iii)] $\Gamma\ba S^{2m-1}$ and $\Gamma'\ba S^{2m-1}$ are $p$-isospectral for all $p$ if and only if $\dim V_{\pi_{{k,p}}}^\Gamma = \dim V_{\pi_{{k,p}}}^{\Gamma'}$ for every $k\in\N$ and every $1\leq p\leq m-1$.
\end{enumerate}
\end{coro}

\section{The work of Ikeda}\label{sec:Ikeda}

In this section we will give a summary of Ikeda's important work on isospectral spherical space forms.
Our notation will somewhat differ from Ikeda's in that we use the language of representation theory introduced in the previous section.

Generating functions are a main tool in Ikeda's work.
One can encode the $0$-spectrum of a space $\Gamma\ba S^{2m-1}$ in the function
$$
F_\Gamma^0(z) = \sum_{k\geq 1}\dim {V_{\pi_{k\varepsilon_1}}^\Gamma}\,z^k.
$$
In light of Corollary~\ref{cor:isosp-dimV^Gamma}~(i), $\Gamma\ba S^{2m-1}$ and $\Gamma'\ba S^{2m-1}$ are $0$-isospectral if and only if $F_{\Gamma}^0(z) = F_{\Gamma'}^0(z)$.
Ikeda proved in \cite[Thm.~2.2]{Ik80b} that $F_{\Gamma}^0(z)$ converges for $|z|<1$ to the rational function
\begin{equation}\label{eq:F^0converges}
F_{\Gamma}^0(z) = \frac{1}{|\Gamma|} \sum_{\gamma\in\Gamma} \frac{1-z^2}{\det(1-z\gamma)}-\frac{1}{|\Gamma|}.
\end{equation}
Here $\det(1-z\gamma)$ stands for $\det(\op{Id}_{2m}-z\gamma)=\prod_\lambda (1-z\lambda)$, where $\lambda$ runs over the eigenvalues of $\gamma$.
Note that $\det(1-z\gamma)=\det(z-\gamma)$ for any $\gamma\in\SO(2m)$, since for any $\lambda$ an eigenvalue of $\gamma$ one has $|\lambda|=1$ and   $\overline\lambda$ is also an eigenvalue.

He observed (see \cite[Corollary~2.3]{Ik80b}) that  \eqref{eq:F^0converges} implies that if $\Gamma$ and $\Gamma'$
are almost conjugate subgroups of $\SO(2m)$ then $\Gamma\ba S^{2m-1}$ and $\Gamma'\ba S^{2m-1}$ are $0$-isospectral.
This result can be viewed as a predecessor of Sunada's method.
In \cite{Ik83}, Ikeda constructed explicitly non-isometric isospectral spherical space forms by using this method.
These pairs are always strongly isospectral and have noncyclic fundamental group.
P.~Gilkey~\cite{Gi} independently found very similar examples.
Later, J.~A.~Wolf~\cite{Wo2} made a step in the determination of all strongly isospectral spherical space forms by using the classification in \cite{Wo}.
In what follows, we will focus our interest on isospectral spherical space forms that are not strongly isospectral.

Ikeda in \cite{Ik88} encoded, for any $p\ge 1$, the $p$-spectrum of a spherical space form $\Gamma\ba S^{2m-1}$ by means of generating functions.
He defined, as a generalization of  $F_\Gamma^0(z)$ the function
\begin{equation}\label{eq:F^p}
  F_\Gamma^p(z) = \sum_{k\geq 0}\dim {V_{\pi_{{k,p+1}}}^\Gamma}\,z^k.
\end{equation}
Although $F_\Gamma^p(z)$ does not have information on all of the $p$-spectrum, by Theorem~\ref{thm:p-spectrum}, the $p$-spectrum is determined by $F_\Gamma^p(z)$ and $F_\Gamma^{p-1}(z)$ together.
In particular, $\Gamma\ba S^{2m-1}$ and $\Gamma'\ba S^{2m-1}$
are $p$-isospectral if and only if $F_{\Gamma}^{p-1}(z) = F_{\Gamma'}^{p-1}(z)$ and $F_{\Gamma}^{p}(z) = F_{\Gamma'}^{p}(z)$ by Corollary~\ref{cor:isosp-dimV^Gamma}~(ii).

He proved, by using a convenient realization of the representation $V_{\pi_{k,p}}$, the following neat formula (see \cite[p.~394]{Ik88}):
\begin{equation}\label{eq:neat}
F_\Gamma^p(z) = (-1)^{p+1} z^{-p} + \frac{1}{|\Gamma|} \sum_{k=0}^p (-1)^{p-k} (z^{k-p}-z^{p-k+2})\sum_{\gamma\in\Gamma} \frac{\chi^{k}(\gamma)}{\det(z-\gamma)}.
\end{equation}
Here, $\chi^k$ denotes the character of the $k$-exterior representation $\bigwedge^k(\C^{2m})$ of $\SO(2m)$.
Set $\widetilde F_\Gamma^k(z)= \sum_{\gamma\in\Gamma} \frac{\chi^{k}(\gamma)}{\det(z-\gamma)}$ for each $0\leq k\leq m-1$.
As a direct consequence of \eqref{eq:neat} he obtains the following result (\cite[Prop.~2.4]{Ik88}).

\begin{prop}\label{prop:p_0-isosp}
Let $p_0\in\Z$,  $0\leq p_0\leq m-1$. Two spherical space forms $\Gamma\ba S^{2m-1}$ and $\Gamma'\ba S^{2m-1}$
are $p$-isospectral for all $0\leq p\leq p_0$ if and only if $\widetilde F_{\Gamma}^p(z) = \widetilde F_{\Gamma'}^p(z)$ for every $0\leq p\leq p_0$.
\end{prop}

By using Proposition~\ref{prop:p_0-isosp}, Ikeda also was able to characterize spherical space forms that are $p$-isospectral for all $p$.
If $w$ is an indeterminate, one can check that $\sum_{k=0}^{2m} (-1)^k \chi^{k}(\gamma) \, w^k= \det(w-\gamma)$, thus
\begin{equation}\label{eq:allpgenerating}
\QQ_\Gamma(w,z):=
\sum_{k=0}^{2m} \;(-1)^k \; \widetilde F_\Gamma^k(z) \; w^k =
\sum_{\gamma\in\Gamma} \frac{\det(w-\gamma)}{\det(z-\gamma)}.
\end{equation}
Therefore one has the following characterization (see \cite[Thm.~2.5]{Ik88}).

\begin{thm}\label{thm:allpIkeda}
Two spherical space forms $\Gamma\ba S^{2m-1}$ and $\Gamma'\ba S^{2m-1}$ are $p$-isospectral for all $p$ if and only if
$\QQ_{\Gamma}(w,z)=\QQ_{\Gamma'}(w,z)$.
\end{thm}

In a similar way as in our comment after \eqref{eq:F^0converges}, Theorem~\ref{thm:allpIkeda} implies that almost conjugate subgroups yield manifolds that are $p$-isospectral for all $p$ (see \cite[Thm.~2.7]{Ik88}).

The previous results are valid for generating functions of arbitrary spherical space forms.
As an application, Ikeda proved the existence of many families of non-isometric $0$-isospectral lens spaces.
Since Pesce~\cite{Pe95} has proved that strongly isospectral lens spaces are necessarily isometric (see also \cite[Prop.~7.2]{LMRhodge}), it turns out that these examples cannot be obtained by Sunada's method.

From now on we will focus on lens spaces, that is, spherical space forms with cyclic fundamental group.
They can be described as follows. For each $q\in\N$ and $s_1,\dots,s_m\in\Z$ coprime to $q$,  denote
\begin{equation}\label{eq3:L(q;s)}
L(q;s_1,\dots,s_m) = \langle\gamma\rangle \ba S^{2m-1}
\end{equation}
where
\begin{equation}\label{eq:gamma}
\gamma=
\diag\left(
\left[\begin{smallmatrix}\cos(2\pi{s_1}/q)&\sin(2\pi{s_1}/q) \\ -\sin(2\pi{s_1}/q)&\cos(2\pi{s_1}/q)
\end{smallmatrix}\right]
,\dots,
\left[\begin{smallmatrix}\cos(2\pi{s_m}/q)&\sin(2\pi{s_m}/q) \\ -\sin(2\pi{s_m}/q)&\cos(2\pi{s_m}/q)
\end{smallmatrix}\right]
\right)
\end{equation}
The element $\gamma$ generates a cyclic group of order $q$ in $\SO(2m)$ that acts freely on $S^{2m-1}$.
The following fact is well known (see \cite[Ch.~V]{Co}).

\begin{prop}\label{prop3:lens-isom}
Let $L=L(q;s_1,\dots,s_m)$ and $L'=L(q;s_1',\dots,s_m')$ be  lens spaces.
Then the following assertions are equivalent.
\begin{enumerate}
  \item $L$ is isometric to $L'$.
  \item $L$ is diffeomorphic to $L'$.
  \item $L$ is homeomorphic to $L'$.
  \item There exist $t \in\Z$ coprime to $q$ and $\epsilon \in \{\pm1\}^{m}$ such that $(s_1,\dots,s_m)$ is a permutation of $(\epsilon_1 t s_1',\dots,\epsilon_m t s_m') \pmod q$.
\end{enumerate}
Furthermore, $L$ and $L'$ are homotopically equivalent if and only if there exists $t\in\Z$ such that $s_1\dots s_m\equiv\pm t^m s_1'\dots s_m'\pmod{q}$.
\end{prop}

Let $L=L(q;s_1,\dots,s_m) = \Gamma\ba S^{2m-1}$ be a lens space and let $\xi=\exp(2\pi i/q)$.
From \eqref{eq:F^0converges}, one has that
\begin{equation}\label{eq:F^0-lens}
F_\Gamma^0(z)=\frac 1q\sum_{l=1}^q \frac{1-z^2}{\prod_{j=1}^m (z-\xi^{s_j l}) (z-\xi^{-s_j l})}-1.
\end{equation}
This formula was first pointed out in \cite[Thm.~3.2]{IY}.

We now sketch Ikeda's construction of families of $0$-isospectral lens spaces.
For each $q$ and $m$ positive integers, he considered the subfamily of lens spaces
\begin{equation}\label{eq:L_0(q;m)}
\widetilde {\mathfrak L}_0(q;m) = \{L(q;s_1,\dots,s_m): s_i\not\equiv \pm s_j\pmod{q}\quad\forall\, i\neq j\}.
\end{equation}
Denote by ${\mathfrak L}_0(q;m)$ the isometry classes in $\widetilde {\mathfrak L}_0(q;m)$.
By the definition, the parameters $s_1,\dots,s_m$ of every lens space in an isometry class in ${\mathfrak L}_0(q;m)$ must be all different.
For $L=L(q;s_1,\dots,s_m) \in \widetilde {\mathfrak L}_0(q;m)$, choose  $h$ integers $\bar s_1,\dots,\bar s_h$ such that
$$
\{\pm s_1,\dots, \pm s_m, \pm \bar s_1,\dots,\pm\bar s_h\}
$$
is a set of representatives of integers mod~$q$, coprime to $q$. Therefore $2m + 2h = \phi(q)$, where $\phi(q)$ denotes the Euler phi function.
Denote by $\bar L$ the $2h-1$-dimensional lens space $L(q;\bar s_1,\dots,\bar s_h)$ and by $\bar\gamma$ the generator of the group $\bar\Gamma$ given by \eqref{eq:gamma} with $\bar{s_i}$ in place
of $s_i$, thus $\bar L=\bar\Gamma\ba S^{2h-1}$.
It is easy to show that two lens spaces $L$ and $L'$ in ${\mathfrak L}_0(q;m)$ are isometric if and only if $\overline{L}$ and $\overline{L'}$ are isometric (\cite[Prop.~3.3]{Ik88}).
Furthermore, Ikeda proved the following important fact.

\begin{prop}\label{prop:duality}
Let $q$ be an odd prime.
Two lens spaces $L,L'\in {\mathfrak L}_0(q;m)$ are $p$-isospectral for all $p$ if and only if $\overline L$ and $\overline {L'}$ are $p$-isospectral for all $p$.
\end{prop}

Ikeda restricted his attention to lens spaces in ${\mathfrak L}_0(q;m)$ for $q$ an odd prime.
In this case, each term ${\prod_{j=1}^m (z-\xi^{s_j l}) (z-\xi^{-s_j l})}$ in \eqref{eq:F^0-lens} divides the $q$-th cyclotomic polynomial $\Phi_q(z):=\prod_{l=1}^{q-1} (z-\xi^l)$ for any $1\leq l\leq q-1$.
Hence, for $L=L(q;s_1,\dots,s_m) = \Gamma\ba S^{2m-1}\in {\mathfrak L}_0(q;m)$, \eqref{eq:F^0-lens} implies that
\begin{equation}\label{eq:F^0-lens_0}
F_\Gamma^0(z)
    = -1+\frac 1q\frac{1-z^2}{(1-z)^{2m}}
    + \frac{1-z^2}{q\,\Phi_q(z)}
    \sum_{l=1}^{q-1} \prod_{j=1}^h (z-\xi^{\bar s_j l}) (z-\xi^{-\bar s_j l}).
\end{equation}
Set
\begin{equation}\label{eq:Psi}
\Psi_\Gamma(z) = \sum_{l=1}^{q-1} \prod_{j=1}^h (z-\xi^{\bar s_j l}) (z-\xi^{-\bar s_j l}).
\end{equation}
This is a polynomial of degree $2h$ with  coefficients in the cyclotomic field $\Q(\xi)$.
Now, \eqref{eq:F^0-lens_0} gives a \emph{finite} condition for $0$-isospectrality, namely, $L,L'\in {\mathfrak L}_0(q;m)$ are $0$-isospectral if and only if $\Psi_\Gamma(z) = \Psi_{\Gamma'}(z)$.

In this way, Ikeda found families of $0$-isospectral lens spaces by showing that, in some cases,  the (well defined) map
$$
\begin{array}{ccc}
L=\Gamma\ba S^{2m-1}\in {\mathfrak L}_0(q;m) &\longmapsto & \Psi_\Gamma(z)\in \Q(\xi)[z]
\end{array}
$$
is not one to one (see \cite[Thm.~3.1]{Ik80a}).
To find such examples, he first compute some coefficients of $\Psi_\Gamma(z)$ (see \cite[Prop.~1.2]{Ik80a}).
Indeed, let $q$ be an odd prime and $q-1=2m+2h$.
For $L =\Gamma\ba S^{2m-1} \in {\mathfrak L}_0(q;m)$,  if we write  $\Psi_{\Gamma}(z) = \sum_{k=0}^{2h} (-1)^k a_k\, z^{2h-k}$,
then he shows that $a_0=q-1$, $a_1=-2m$, $a_2=m(q-2m+1)$ and $a_k=a_{2h-k}$ for all $0\leq k\leq 2h$.
Note that, $a_0$, $a_1$, $a_2$, $a_{2h-2}$, $a_{2h-1}$ and $a_{2h}$ do not depend on $L$, thus, if $h=2$, $\Psi_\Gamma(z)$ is the same for all lens spaces.
As a consequence he obtained the following result (\cite[Thm.~3.1]{Ik80a}).

\begin{thm}\label{thm:0-isosp-flias}
Let $q$ be an odd prime and let $m$ be such that $2m + 4=q-1$ (i.e.~ $m=(q-5)/2$).
Then, any two lens spaces in $\mathfrak L_0(q,m)$ are $0$-isospectral.
\end{thm}

The previous theorem gives a way to obtain increasing families of pairwise $0$-isospectral lens spaces of dimension $n\ge 5$, where $q$ runs over the odd prime numbers.
Indeed, $|\mathfrak L_0(q,m)|\geq\frac{1}{(q-1)/2}\binom{(q-1)/2}{m}=\frac{q-3}{4}$ under the hypotheses in Theorem~\ref{thm:0-isosp-flias}.
The simplest case is $q=11$, thus $m=3$ and dimension $n=5$.
Then one can check that ${\mathfrak L}_0(q;3)$ has two isometry classes represented by $L(11;1,2,3)$ and $L(11;1,2,4)$.
One can also check that they are homotopically equivalent to each other by Proposition~\ref{prop3:lens-isom}.
However, if one takes  $q=13$, then $m=4$, $n=7$ and ${\mathfrak L}_0(q;4)$ contains $L(13;1,2,3,4)$, $L(13;1,2,3,5)$ and $L(13;1,2,3,6)$, which are not homotopically equivalent to each other.

Ikeda proved that two non-isometric lens spaces in ${\mathfrak L}_0(q;m)$ as in Theorem~\ref{thm:0-isosp-flias} cannot
be $p$-isospectral for all $p$ (see \cite[Thm.~3.9]{Ik88}).
The argument is as follows.
Suppose that $L,L'\in {\mathfrak L}_0(q;m)$ are $p$-isospectral for all $p$, with $q$ an odd prime and $q-1=2m+4$.
By Proposition~\ref{prop:duality}, $\overline L$ and $\overline{L'}$ are $p$-isospectral for all $p$ and of dimension $2h-1=3$.
However, two $3$-dimensional $0$-isospectral lens spaces must be isometric (see \cite{IY}, \cite{Y}), thus $\overline L$ and $\overline{L'}$, and therefore $L$ and $L'$, are isometric.

However, by using the same family as in Theorem~\ref{thm:0-isosp-flias}, Ikeda in \cite{Ik88} found for each $p_0\ge 0$, examples of pairs of lens spaces that are $p$-isospectral for every $0\leq p\leq p_0$ but are not $p_0+1$-isospectral.
We conclude this section by giving the main ideas used in his proof.
We will use the condition of $p$-isospectrality for $0\leq p\leq p_0$ in Proposition~\ref{prop:p_0-isosp}.
Let $L=L(q;s_1,\dots,s_m)=\Gamma\ba S^{2m-1}\in {\mathfrak L}_0(q;m)$, where $q$ is an odd prime number and $q-1=2m+4$.
Similarly as in \eqref{eq:F^0-lens_0} we obtain that
\begin{align}\label{eq:F^p-lens0}
\widetilde F_\Gamma^p(z)
    &= \frac{\tbinom{2m}{p}}{(z-1)^{2m}} +
        \sum_{l=1}^{q-1} \frac{\chi^p(g^l)}{\prod_{j=1}^m (z-\xi^{s_j l}) (z-\xi^{-s_j l})}  \\
    &= \frac{\tbinom{2m}{p}}{(z-1)^{2m}}-\Phi_q(z)^{-1} (z-1)^4 \notag\\
    &\qquad +\Phi_q(z)^{-1} \sum_{l=1}^q \chi^p(g^l)\prod_{j=1}^2 (z-\xi^{\bar s_j l}) (z-\xi^{-\bar s_j l}).\notag
\end{align}
Hence, the polynomial
$$
\Psi_\Gamma^p(z) := \sum_{l=1}^q \chi^p(g^l)\prod_{j=1}^2 (z-\xi^{\bar s_j l}) (z-\xi^{-\bar s_j l}),
$$
which has degree four, determines $\widetilde F_\Gamma^p(z)$, thus the five coefficients of $\Psi_\Gamma^p(z)=\sum_{t=0}^{4} (-1)^t b_{L,p}^{(t)}\, z^t$ play an important role.
One can check that (see \cite[p.~404]{Ik88})
\begin{align*}
  b_{L,p}^{(0)}&=b_{L,p}^{(4)}= -1+ q \textstyle\sum\limits_{d=0}^{[p/2]} \tbinom{m}{d} \; A_{L}^{(p-2d)}(0),\\
  b_{L,p}^{(1)}&=b_{L,p}^{(3)}=-1+ 2q \textstyle\sum\limits_{d=0}^{[p/2]} \tbinom{m}{d} \; \left(A_{L}^{(p-2d)}(\bar s_1)+A_{L}^{(p-2d)}(\bar s_2)\right) ,\\
  b_{L,p}^{(2)}&=-1+2b_{L,p}^{(0)}+ 2q\textstyle\sum\limits_{d=0}^{[p/2]} \tbinom{m}{d} \; \left( A_{L}^{(p-2d)}(\bar s_1+\bar s_2) + A_{L}^{(p-2d)}(\bar s_1-\bar s_2)\right),
\end{align*}
where
\begin{equation}\label{eq:A_l^(k)(s)}
A_{L}^{(k)}(s) := \#\left\{A\subset S:
\begin{array}{l}
a\not\equiv -a'\pmod q\quad\forall a\neq a'\in A,\\
|A|=k,\; \sum_{a\in A}a\equiv s\pmod{q}
\end{array}
\right\}.
\end{equation}
and $S=\{\pm s_1,\dots,\pm s_m\}$.
Now, Proposition~\ref{prop:p_0-isosp} can be rewritten in the particular case of lens spaces in ${\mathfrak L}_0(q;m)$ as follows (see \cite[Prop.~4.2]{Ik88}).

\begin{prop}
Let $q$ be an odd prime, $m=(q-5)/2$ and $0\leq p_0\leq m-1$.
Then, two lens spaces $L=L(q;s_1,\dots,s_m)$ and $L'=L(q;s_1',\dots,s_m')$ in ${\mathfrak L}_0(q;m)$ are $p$-isospectral for all $0\leq p\leq p_0$ if and only if
\begin{equation}\label{eq:0--p-condition}
\left\{
\begin{array}{l}
A_{L}^{(p)}(0) = A_{L'}^{(p)}(0),\\
A_{L}^{(p)}(\bar s_1) +A_{L}^{(p)}(\bar s_2)= A_{L'}^{(p)}(\bar s_1') + A_{L'}^{(p)}(\bar s_2'),\\
A_{L}^{(p)}(\bar s_1+\bar s_2) +A_{L}^{(p)}(\bar s_1-\bar s_2)= A_{L'}^{(p)}(\bar s_1' + \bar s_2') + A_{L'}^{(p)}(\bar s_1'-\bar s_2'),\\
\end{array}
\right.
\end{equation}
for all $0\leq p\leq p_0$.
\end{prop}

Ikeda found subfamilies in ${\mathfrak L}_0(q;m)$ such that satisfy \eqref{eq:0--p-condition}.
Set
\begin{equation}
{\mathfrak L}_p(q;m) = \left\{L(q;s)\in {\mathfrak L}_0(q;m): \begin{array}{l}
a_1\bar s_1+a_2\bar s_2\not\equiv0\pmod q ,\\
\forall \; 1\leq |a_1|+|a_2|\leq p+2
\end{array}
\right\},
\end{equation}
thus one has the filtration
\begin{equation}\label{eq:filtration}
{\mathfrak L}_0(q;m) \supset {\mathfrak L}_1(q;m) \supset {\mathfrak L}_2(q;m)\supset\dots\,.
\end{equation}
For example, for $q\geq11$ an odd prime, if $\overline{L}=L(q;1,2)$ then $L\in {\mathfrak L}_0(q;m)\smallsetminus {\mathfrak L}_1(q;m)$ and if $\overline{L}=L(q;1,3)$ then $L\in {\mathfrak L}_1(q;m)\smallsetminus {\mathfrak L}_2(q;m)$.

By making several computations with the numbers in \eqref{eq:A_l^(k)(s)}, he showed that two lens spaces at the same level $p_0$ of the filtration satisfy \eqref{eq:0--p-condition} for all $0\leq p\leq p_0$.
Moreover, if only one of them lies in the next level $p_0+1$, then they cannot satisfy \eqref{eq:0--p-condition} for $p=p_0+1$.
More precisely, we can now state \cite[Thm.~4.1]{Ik88}.

\begin{thm}\label{thm:mainIkeda}
Let $q$ be an odd prime, $m=(q-5)/2$ and $0\leq p_0\leq m-1$ and let $L$ and $L'$ be lens spaces in ${\mathfrak L}_{p_0}(q;m)$.
Then $L$ and $L'$ are $p$-isospectral for all $0\leq p\leq p_0$.
If furthermore $L\in {\mathfrak L}_{p_0+1}(q;m)$ and $L'\notin {\mathfrak L}_{p_0+1}(q;m)$, then $L$ and $L'$ are not $p_0+1$-isospectral.
\end{thm}

Now we fix $p_0\geq0$.
Let $q$ be a prime number greater than $(p_0+2)(p_0+3)+1$ and set $m=(q-5)/2$.
If $L$ and $L'$ are the lens space in ${\mathfrak L}_{0}(q;m)$ such that $\overline L=L(q;1,p_0+2),\,\overline{L'}=L(q;1,p_0+3)\in {\mathfrak L}_{0}(q;2)$, then one can check that $L\in {\mathfrak L}_{p_0}(q;m)\smallsetminus {\mathfrak L}_{p_0+1}(q;m)$ and $L'\in {\mathfrak L}_{p_0+1}(q;m)$.
As a consequence one obtains the following corollary (\cite[Thm.~4.10]{Ik88}).

\begin{coro}\label{thm:mainIkeda2}
For each $p_0\geq0$ there are lens spaces that are $p$-isospectral for all $0\leq p\leq p_0$ but are not $p_0+1$-isospectral.
\end{coro}

\section{Isospectral lens spaces and $\norma{\cdot}$-isospectral lattices}

In this section we will explain the remarkable relation between isospectrality of lens spaces and isospectrality of integral lattices with respect to the one-norm, introduced in \cite{LMRhodge}.
Using this connection we were able to find examples of lens spaces $p$-isospectral for all $p$ which are not coming from Sunada's method, and  are far from being strongly isospectral.
Indeed, we present an infinite family of pairs of $5$-dimensional lens spaces with these properties,
and as a byproduct, an infinite family of such pairs in increasing dimensions.
Finally, by means of a \emph{finite-implies-infinite} principle (see \cite[\S4]{LMRhodge}), we find ---with the help of the computer--- many more such pairs in low dimensions.

We first present the main ideas and the results without proofs, so that the reader can get to them quickly, and after this we develop the mathematical arguments supporting these results.
We naturally associate to a lens space $L=L(q;s_1,s_2,\dots,s_m)$ of dimension $2m-1$ the integral lattice $\mathcal L=\mathcal L(q;s_1,\dots,s_m)$ of rank $m$ given by the congruence equation
\begin{equation}\label{eq:mathcalL}
   (a_1,\dots,a_m) \in\mathcal L \iff  a_1s_1+\dots +a_ms_m\equiv0\pmod q.
\end{equation}
We call a lattice of this kind a \emph{congruence lattice}.
We consider the one norm, i.e.\ $\norma{(a_1,\dots,a_m)}:=|a_1|+\dots+|a_m|$.
By using Proposition~\ref{prop3:lens-isom} it is easy to prove that \emph{two lens spaces are isometric if and only if their associated congruence lattices are $\norma{\cdot}$-isometric} (see \cite[Prop.~3.3]{LMRhodge}).
It was surprising to discover that the isospectrality of two lens spaces is directly connected with the isospectrality in one-norm of the associated lattices (see \cite[Thm.~3.9(i)]{LMRhodge}).

\begin{thm}\label{thm:0-characterization}
Two lens spaces are $0$-isospectral if and only if the associated congruence lattices are $\norma{\cdot}$-isospectral.
\end{thm}

If one considers one individual $p$ only, $p$-isospectrality for two lens spaces,  does not correspond to a clean and neat condition on the associated lattices, as in the previous theorem for $p=0$.
However, the condition of being $p$-isospectral for all $p$ simultaneously turns out to correspond ---again as a happy surprise--- to a nice geometric condition on the associated lattices, which we call \emph{$\norma{\cdot}^*$-isospectrality}:
for each $k$ and $\zz$, both lattices must have the same number of vectors with one-norm equal to $k$ and $\zz$ zero coordinates.

\begin{thm}\label{thm:p-characterization}
Two lens spaces are $p$-isospectral for all $p$ if and only if the associated congruence lattices are $\norma{\cdot}^*$-isospectral.
\end{thm}

The proofs of these theorems use representation theory of compact Lie groups and properties of the weight lattice.
The ideas are given in the next subsection.

The basic example of $\norma{\cdot}^*$-isospectral congruence lattices is the pair
\begin{equation}\label{eq:basicexample}
\mathcal L (49;\,1,\,6,\,15) \quad\textrm{ and } \quad\mathcal L(49;\,1,\,6,\,20).
\end{equation}
These 3-dimensional lattices produce two 5-dimensional non-isometric lens spaces which in light of the previous theorem are $p$-isospectral for every $p$.
This example is the first one of the following infinite family of pairs
\begin{equation}\label{eq:infinitefamily0}
\mathcal L(r^2t;\,1,\,rt-1,\,2rt+1)  \quad\textrm{ and }\quad \mathcal  L(r^2t;\,1,\,rt-1,\,3rt-1),
\end{equation}
for $r,t\in\N$ with $r$ not divisible by~$3$ (see \cite[Thm.~6.3]{LMRhodge}).
For some purposes, it will be convenient to write this in the following equivalent way
\begin{equation}\label{eq:infinitefamily}
\mathcal L(r^2t;\,1,\,1+rt,\,1+3rt)  \quad\textrm{ and }\quad \mathcal  L(r^2t;\,1,\,1-rt,\,1-3rt),
\end{equation}

Fraom these examples, it is possible to construct examples in arbitrarily high dimensions by using Proposition~\ref{prop:dualityq=r^2}.
In this way, we obtain for each pair of $(2m-1)$-dimensional lens spaces in our family, another pair of $(2h-1)$-dimensional lens spaces with $2h+2m=\phi(q)$, $q=r^2t$, that are again $p$-isospectral for every $p$.
For example when $t=1$ and $r$ is prime, the dimension increases from 5 to $2h-1=r^2-r-7$.
For the basic pair when $r=7$ one has that $h=18$ and $2h-1= 35$.

These examples are not the only existing ones, as one can guess.
We proved in \cite[Thm.\ 4.2]{LMRhodge} that to check $\norma{\cdot}^*$-isospectrality, it suffices to check it in a finite cube, which means that only finitely many computations are enough to ensure $p$-isospectrality for all $p$ of the lens spaces.
By using this, with the help of a computer, we found many more examples.
Moreover, one can find \emph{all} the existing examples for  given fixed $m$ and $q$ (see Tables~\ref{table:n=5}, \ref{table:n=7} and \ref{table:n=9}).
However, the computing time grows rapidly with $m$.

\subsection{Characterization theorems}

In this subsection we give the ideas leading to Theorems~\ref{thm:0-characterization} and \ref{thm:p-characterization} and sketch their proofs. The approach is based on representation theory of compact Lie groups.

From Section~\ref{sec:specspherforms},  $G=\SO(2m)$, $K=\SO(2m-1)$, $T$ is the standard maximal torus of $G$, and $P(G)=\bigoplus_{j=1}^m \Z\varepsilon_j\simeq \Z^m$ is the weight lattice of $G$.
As we have seen in Theorem~\ref{thm:tau-spectrum}, the $\tau$-spectrum of $\Gamma\ba S^{2m-1}$ is determined by the numbers $\dim V_\pi^\Gamma$ for every $\pi\in\widehat G$ such that $[\tau:\pi]=1$.
Any $\pi\in\widehat G$ decomposes as a sum of weight spaces under the action of $T$ as
$$
V_\pi=\sum_{\eta\in P(G)} V_\pi(\eta).
$$
The multiplicity of a weight $\eta\in P(G)$ in $\pi$ is $m_\pi(\eta):=\dim V_\pi(\eta)$.
If $\Gamma\subset T$, it follows that
\begin{equation}\label{eq:dimV_pi}
\dim V_\pi^\Gamma = \sum_{\eta\in P(G)} \dim V_\pi(\eta)^\Gamma = \sum_{\eta\in \mathcal L_\Gamma} m_\pi(\eta)
\end{equation}
where $\mathcal L_\Gamma=\{\eta\in P(G): \gamma^\eta=1\quad\forall\,\gamma\in\Gamma\}$, which is a sublattice of $P(G)\simeq\Z^m$ depending only on $\Gamma$ but not on $\pi$.
Here $\gamma^\eta$ denotes the scalar for which $\gamma$ acts on $V_\pi(\eta)$.

A lens space $L(q;s_1,\dots,s_m)=\Gamma\ba S^{2m-1}$ satisfies that $\Gamma\subset T$ since it is generated by $\gamma\in T$ as in \eqref{eq:gamma}.
Since $\gamma^\eta = e^{2\pi i\frac{a_1s_1+\dots+a_ms_m}{q}}$ for
$\eta=\sum_ja_j\varepsilon_j\in P(G)$,
we have that $\mathcal L_\Gamma = \mathcal L(q;s_1,\dots,s_m)$ defined in \eqref{eq:mathcalL}.

\begin{proof}[Sketch of proof of Theorem~\ref{thm:0-characterization}.]
One can show that (see \cite[Lemma~3.6]{LMRhodge}), when $\pi=\pi_{k\varepsilon_1}$ is the irreducible representation of $\SO(2m)$ with highest weight $k\varepsilon_1$, the multiplicity of $\eta\in\Z^m$ in $\pi$ is
\begin{equation}\label{eq:mult-ke_1}
m_{\pi_{k\varepsilon_1}}(\eta) =
\begin{cases}
\binom{r+m-2}{m-2} & \text{ if }\, \norma{\eta}=k-2r \;\text{ with } r\in \N_0,\\
0 & \text{ otherwise.}
\end{cases}
\end{equation}
In particular, $m_{\pi_{k\varepsilon_1}}(\eta)$ depends only on $\norma{\eta}$.
From \eqref{eq:dimV_pi} and \eqref{eq:mult-ke_1} it follows that
\begin{align}
\dim V_{\pi_{k\varepsilon_1}}^\Gamma
    &= \sum_{r=0}^{[k/2]} \sum_{\eta\in \mathcal L:\atop \norma{\eta}=k-2r} \tbinom{r+m-2}{m-2} \label{eq:mult-tau_0}\\
    &= \sum_{r=0}^{[k/2]}\tbinom{r+m-2}{m-2} \;\#\{\eta\in\mathcal L_\Gamma: \norma{\eta}=k-2r\}.
    \notag
\end{align}
Moreover, by Theorem~\ref{thm:p-spectrum}, this number is exactly the multiplicity of the eigenvalue $\lambda_k=k^2+k(2m-2)$ of the Laplace-Beltrami operator $\Delta_{\tau_0,\Gamma}$ on $\Gamma\ba S^{2m-1}$.
This clearly shows that two lens spaces are $0$-isospectral if their associated lattices are $\norma{\cdot}$-isospectral, thus proving the converse assertion in Theorem~\ref{thm:0-characterization}.
The remaining implication is proved by induction on $k$ (see \cite[Thm.~3.9(i)]{LMRhodge}).
\end{proof}

\begin{proof}[Sketch of proof of Theorem~\ref{thm:p-characterization}.]
We proceed as in the previous theorem but in this case there are more difficulties.
By Corollary~\ref{cor:isosp-dimV^Gamma}~(iii), we have to show that $\dim V_{\pi_{k,p}}^{\Gamma}=\dim V_{\pi_{k,p}}^{\Gamma'}$ for all $k$ and $p$, where $\pi_{k,p}$ is as in Section~\ref{sec:specspherforms}.
Now, $\dim V_{\pi_{k,p}}^{\Gamma} = \sum_{\eta \in \mathcal L} m_{\pi_{k,p}}(\eta)$, but in this case we do not know an explicit formula like \eqref{eq:mult-ke_1} for $m_{\pi_{k,p}}(\eta)$ for arbitrary $k$ and $p$.
Fortunately, this difficulty could be overcome thanks to the following cute regularity property of the multiplicities:
\begin{quote}
\emph{two weights with the same one-norm and the same number of zero coordinates have the same multiplicity}.
\end{quote}
This was proved in \cite[Lem.~3.7]{LMRhodge} with techniques of representation theory of compact Lie groups.

Analogously to \eqref{eq:mult-ke_1}, by using the above property we obtain that (see \cite[Thm.~3.8]{LMRhodge})
\begin{equation}\label{eq:dim V_k,p^Gamma}
\dim V_{\pi_{k,p}}^{\Gamma}=
    \sum_{r=0}^{[(k+p)/2]} \;\sum_{\zz =0}^m \; m_{\pi_{k,p}}(\mu_{r,\zz })\; N_{\mathcal L}(k+p-2r,\zz ),
\end{equation}
where $\mu_{r,\zz }$ is any weight in $\mathcal L$ with $\norma{\mu_{r,\zz }} = k+p-2r$ and having $\zz$ zero coordinates, and $N_{\mathcal L}(r,\zz )$ denotes the number of weights with $\zz$ zero coordinates and one-norm equal to $r$.
Now, clearly, the converse of Theorem~\ref{thm:p-characterization} follows.
The other assertion can again be proved by induction on $k$ (see \cite[Thm.~3.9(ii)]{LMRhodge}).
\end{proof}

\subsection{Construction of $\norma{\cdot}^*$-isospectral lattices}\label{subsec:construction}
The characterization of lens spaces $p$-isospectral for all $p$ given in Theorem~\ref{thm:p-characterization},
motivated us to look for examples of $\norma{\cdot}^*$-isospectral congruence lattices.
It seems interesting that one can work on the construction of such examples by just working on lattices,
without any use of differential geometry.

For $r,t\in\N$, $r>1$ not divisible by $3$, we set $q=r^2t$ and consider the lattices in \eqref{eq:infinitefamily},
$\mathcal L = \mathcal L(q; 1,1+rt,3rt+1)$ and $\mathcal L' = \mathcal L(q; 1,1-rt,1-3rt)$.
This is an infinite two-parameter family of pairs of $\norma{\cdot}^*$-isospectral lattices in $\Z^m$
for $m=3$ (\cite[Thm.\ 6.3]{LMRhodge}).
For $r\ge 7$ they are not $\norma{\cdot}$-isometric (see \cite[Lemma 5.4]{LMRhodge}).
We note that the dimension $m=3$ of these examples is minimal, since
Ikeda and Yamamoto showed that such pairs cannot exist in dimension $m=2$ (\cite{IY}, \cite{Y}).

The first step in the proof is to reduce the problem to show that the lattices are just $\norma{\cdot}$-isospectral, since one can verify that, for $1\leq \zz\leq 3$, the number of elements in $\mathcal L$ and $\mathcal L'$ with $\zz$
zero coordinates and a fixed one-norm coincide (see \cite[Lemma 6.1]{LMRhodge}).
This implies, for the family of pairs $L=L(r^2t; 1,1+rt,1+3rt), L'=L(r^2t; 1,1-rt,1-3rt)$, the pleasant fact in spectral geometry that:
\emph{$L$ and $L'$ are $p$-isospectral for all $p$ if and only if $L$ and $L'$ are $0$-isospectral}.

According to the previous paragraph, it is sufficient to check that $\mathcal L$ and $\mathcal L'$ are $\norma{\cdot}$-isospectral.
More precisely, $N_{\mathcal L}(k,\zz) = N_{\mathcal L'}(k,\zz)$, where $N_{\mathcal L}(k,\zz)$ denotes the number of $\eta\in\mathcal L$ with $\norma{\eta}=k$ and $\zz$ zero coordinates.
By a careful calculation of these numbers, we check they coincide.

\subsection{Examples in arbitrarily large dimensions}
We will show that the infinite family of pairs in dimension $5$ given in the previous section allows to produce an infinite family of pairs in arbitrarily large dimensions. For this, we prove an extension of Proposition~\ref{prop:duality} for $q=r^2$, $r$ prime.

We recall from \eqref{eq:L_0(q;m)} that $\mathfrak L_0(q,m)$ stands for the set of lens spaces of dimension $2m-1$, fundamental group of order $q$ and different parameters.
For $L=\Gamma\ba S^{2m-1}$, the function $\QQ_L(w,z):=\QQ_\Gamma(w,z)$ given in \eqref{eq:allpgenerating} characterizes all $p$-spectrum.
If $L=L(q;s_1,\dots,s_m)\in\mathfrak L_0(q,m)$ then
$
\QQ_L(w,z)=\sum_{k=0}^{q-1} \frac{\det(w-\gamma^k)}{\det(z-\gamma^k)}
$
since $\Gamma =\{\gamma^k: 0\leq k\leq q-1\}$.

\begin{prop}\label{prop:dualityq=r^2}
Let $q=r^2$ with $r$ prime and let $L=L(q;s_1,\dots,s_m)$ and $L'=(q;s_1,\dots,s_m)$ be lens spaces in $\mathfrak L_0(q,m)$ such that $s_j\equiv\pm1\pmod r$ and $s_j'\equiv \pm1\pmod r$ for all $j$. Then, $L$ and $L'$ are $p$-isospectral for all $p$ if and only if $\overline L$ and $\overline{L'}$ are $p$-isospectral for all $p$.
\end{prop}

\begin{proof}
We will obtain in \eqref{eq:QQ} a useful relation connecting $\QQ_L(w,z)$ and $\QQ_{\bar L}(w,z)$ for $L=L(q;s_1,\dots,s_m)\in \mathfrak L_0(q,m)$ such that $s_j\equiv \pm1\pmod r$ for all $j$.
Here $\bar\gamma$ and $\bar s_1,\dots,\bar s_h$ are as in \S \ref{sec:Ikeda}, the paragraph before Prop.~\ref{prop:duality}.
We have
\begin{equation}\label{eq:Q1}
\QQ_L(w,z) =
\sum_{l=0}^{r-1} \frac{\det\big(w-{(\gamma^r)}^l\big)}{\det\big(z-{(\gamma^r)}^l\big)} +  \sum_{\gcd(k,r)=1} \frac{\det(w-\gamma^k)}{\det(z-\gamma^k)}.
\end{equation}
The eigenvalues of $\gamma^{r}$ are $e^{\pm2\pi i rs_j /r^2}=e^{\pm2\pi i /r}$ for $1\leq j\leq m$ since $r_j\equiv\pm1\pmod r$, thus
$
\sum_{l=0}^{r-1} \frac{\det(w-\gamma^{lr})}{\det(z-\gamma^{lr})} =
\QQ_{L_0}(w,z)
$
where $L_0$ stands for $\langle\gamma^r\rangle\ba S^{2m-1}$.
One can check that $L_0$ is isometric to $L(r;{1,\dots,1})$, hence $\QQ_{L_0}(w,z)$ does not depend on $L$.

Since $r$ is prime, $\det(z-\gamma^k)\det(z-\bar \gamma^k)$ is equal to $\Phi_{q}(z)$ if $\gcd(k,q)=1$, to $\Phi_{r}(z)^r$ if $\gcd(k,q)=r$, and to $(z-1)^{r^2-r}$ if $\gcd(k,q)=r^2$.
Hence,
\begin{multline}\label{eq:QQ}
\QQ_L(w,z)
=  \QQ_{L_0}(w,z) +  \frac{\Phi_{q}(w)}{\Phi_{q}(z)} \sum_{\gcd(k,r)=1} \frac{\det(z-\bar\gamma^k)}{\det(w-\bar\gamma^k)}=\\
=  \QQ_{L_0}(w,z) +
    \frac{\Phi_{q}(w)}{\Phi_{q}(z)} \left(\QQ_{\overline{L\!}}\,(z,w)-\frac{(z-1)^{2h}}{(w-1)^{2h}}-\sum_{l=1}^{r-1} \frac{\det\big(z-{(\bar\gamma^r)}^l\big)}{\det\big(w-{(\bar\gamma^r)}^l)\big)}\right)=\\
= \! \QQ_{L_0}(w,z) +
    \frac{\Phi_{q}(w)}{\Phi_{q}(z)} \!\left(\!\QQ_{\overline{L\!}}\,(z,w)\!-\!\frac{(z-1)^{2h}}{(w-1)^{2h}}\!-\!
    \frac{\Phi_{r}(z)^r}{\Phi_{r}(w)^{r}}
    \left(\!\QQ_{L_0}(w,z)-\frac{(w-1)^{2m}}{(z-1)^{2m}}\right)\!\right)\!.
\end{multline}
Thus, the last expression for $\QQ_{L}(w,z)$ involves $\QQ_{\bar L}(z,w)$ and other functions which do not depend on $L$.
This clearly shows that $\QQ_{L}(w,z)$ and $\QQ_{\overline L}(w,z)$ determine each other in this case.
In particular, $\QQ_{L}(w,z)=\QQ_{L'}(w,z)$ if and only if $\QQ_{\overline{L}}(w,z)=\QQ_{\overline{L'}}(w,z)$,
thus the assertion follows from Theorem~\ref{thm:allpIkeda}.
\end{proof}

As a corollary we can now state (\cite[Thm.~7.3]{LMRhodge})

\begin{thm}\label{thm:largedimension}
For any $n_0 \ge 5$, there exist pairs of non-isometric lens spaces of dimension $n$, with $n\geq n_0$, that are $p$-isospectral for all $p$.
\end{thm}

\begin{proof}
For each odd prime $r\geq7$ set $t=1$ and $q=r^2$.
The corresponding $5$-dimensional lens spaces $L,L'\in\mathfrak L_0(q;3)$ from \eqref{eq:infinitefamily} are $p$-isospectral for all $p$, by Theorem~\ref{thm:p-characterization}.
By Proposition~\ref{prop:dualityq=r^2}, $\overline L$ and $\overline {L'}$ are $p$-isospectral for all $p$ and have dimension $2h-1=\phi(r^2)-7 = r^2-r-7$. This quantity tends to infinity when $r$ does, thus the assertion in the theorem follows.
\end{proof}

We recall from \cite[Thm.~7.3]{LMRhodge} that, by using our $5$-dimensional examples, one can construct, in
\emph{every} dimension $n\geq5$, pairs of $n$-dimensional Riemannian manifolds that are $p$-isospectral for all $p$ and are not strongly isospectral.

\begin{rem}
We are interested in the question whether one can extend the duality property in Proposition~\ref{prop:dualityq=r^2}
for more general values of $q$.
(We have checked that all the pairs dual to the pairs in the tables remain $p$-isospectral for all $p$.)
\end{rem}

\subsection{Computations and tables}

\begin{table}
\caption{Pairs of lens spaces $p$-isospectral for all $p$ of dimension $n=5$ and fundamental group of order $q\leq 500$.}\label{table:n=5}
\begin{tabular}[t]{cccr@{,}r@{,}r}
$q$& $r$ & $t$ & $d_0$ & $d_1$ & $d_2$\\ \hline
 49 &  7 & 1 & 0& 1& 3 \\
 64 &  8 & 1 & 0& 1& 3 \\
 98 &  7 & 2 & 0& 1& 3 \\
100 & 10 & 1 & 0& 1& 3 \\
100 & 10 & 1 & 0& 1& 4 \\
121 & 11 & 1 & 0& 1& 3 \\
121 & 11 & 1 & 0& 1& 4 \\
121 & 11 & 1 & 0& 1& 5 \\
121 & 11 & 1 & 0& 2& 5 \\
121 & 11 & 1 & 0& 2& 6 \\
128 &  8 & 2 & 0& 1& 3 \\
147 &  7 & 3 & 0& 1& 3 \\
169 & 13 & 1 & 0& 1& 3 \\
169 & 13 & 1 & 0& 1& 4 \\
169 & 13 & 1 & 0& 1& 5 \\
169 & 13 & 1 & 0& 1& 6 \\
169 & 13 & 1 & 0& 2& 5 \\
169 & 13 & 1 & 0& 2& 6 \\
169 & 13 & 1 & 0& 2& 7 \\
169 & 13 & 1 & 0& 3& 7 \\
192 &  8 & 3 & 0& 1& 3 \\
196 & 14 & 1 & 0& 1& 3 \\
196 & 14 & 1 & 0& 1& 4 \\
196 & 14 & 1 & 0& 1& 5 \\
196 & 14 & 1 & 0& 1& 6 \\
196 & 14 & 1 & 0& 2& 5 \\
196 & 14 & 1 & 0& 2& 6 \\
196 &  7 & 4$\dag$& 0& 1& 3 \\
196 & 14 & 1 & 0& 3& 8 \\
200 & 10 & 2 & 0& 1& 3 \\
200 & 10 & 2 & 0& 1& 4 \\
242 & 11 & 2 & 0& 1& 3 \\
242 & 11 & 2 & 0& 1& 4 \\
242 & 11 & 2 & 0& 1& 5 \\
242 & 11 & 2 & 0& 2& 5 \\
242 & 11 & 2 & 0& 2& 6 \\
245 &  7 & 5 & 0& 1& 3 \\
256 & 16 & 1 & 0& 1& 3 \\
256 & 16 & 1 & 0& 1& 6 \\
256 & 16 & 1 & 0& 1& 7 \\
256 & 16 & 1 & 0& 2& 5 \\
256 & 16 & 1 & 0& 2& 6 \\
256 &  8 & 4$\dag$& 0& 1& 3 \\
256 & 16 & 1 & 0& 2& 7 \\
256 & 16 & 1 & 0& 3& 9 \\
289 & 17 & 1 & 0& 1& 3 \\
289 & 17 & 1 & 0& 1& 4 \\
289 & 17 & 1 & 0& 1& 5 \\
289 & 17 & 1 & 0& 1& 6 \\
289 & 17 & 1 & 0& 1& 7 \\
289 & 17 & 1 & 0& 1& 8 \\
289 & 17 & 1 & 0& 2& 5 \\
289 & 17 & 1 & 0& 2& 6 \\
\end{tabular}
\quad
\begin{tabular}[t]{cccr@{,}r@{,}r}
$q$& $r$ & $t$ & $d_0$ & $d_1$ & $d_2$\\ \hline
289 & 17 & 1 & 0& 2& 7 \\
289 & 17 & 1 & 0& 2& 8 \\
289 & 17 & 1 & 0& 2& 9 \\
289 & 17 & 1 & 0& 3& 7 \\
289 & 17 & 1 & 0& 3& 8 \\
289 & 17 & 1 & 0& 3& 9 \\
289 & 17 & 1 & 0& 4& 9 \\
289 & 17 & 1 & 0& 4&10 \\
294 &  7 & 6 & 0& 1& 3 \\
300 & 10 & 3 & 0& 1& 3 \\
300 & 10 & 3 & 0& 1& 4 \\
320 &  8 & 5 & 0& 1& 3 \\
324 & 18 & 1 & 0& 1& 5 \\
324 & 18 & 1 & 0& 1& 8 \\
324 & 18 & 1 & 0& 2& 7 \\
338 & 13 & 2 & 0& 1& 3 \\
338 & 13 & 2 & 0& 1& 4 \\
338 & 13 & 2 & 0& 1& 5 \\
338 & 13 & 2 & 0& 1& 6 \\
338 & 13 & 2 & 0& 2& 5 \\
338 & 13 & 2 & 0& 2& 6 \\
338 & 13 & 2 & 0& 2& 7 \\
338 & 13 & 2 & 0& 3& 7 \\
343 &  7 & 7 & 0& 1& 3 \\
361 & 19 & 1 & 0& 1& 3 \\
361 & 19 & 1 & 0& 1& 4 \\
361 & 19 & 1 & 0& 1& 5 \\
361 & 19 & 1 & 0& 1& 6 \\
361 & 19 & 1 & 0& 1& 7 \\
361 & 19 & 1 & 0& 1& 8 \\
361 & 19 & 1 & 0& 1& 9 \\
361 & 19 & 1 & 0& 2& 5 \\
361 & 19 & 1 & 0& 2& 6 \\
361 & 19 & 1 & 0& 2& 7 \\
361 & 19 & 1 & 0& 2& 8 \\
361 & 19 & 1 & 0& 2& 9 \\
361 & 19 & 1 & 0& 2&10 \\
361 & 19 & 1 & 0& 3& 7 \\
361 & 19 & 1 & 0& 3& 8 \\
361 & 19 & 1 & 0& 3& 9 \\
361 & 19 & 1 & 0& 3&10 \\
361 & 19 & 1 & 0& 4& 9 \\
361 & 19 & 1 & 0& 4&10 \\
361 & 19 & 1 & 0& 4&11 \\
361 & 19 & 1 & 0& 5&11 \\
363 & 11 & 3 & 0& 1& 3 \\
363 & 11 & 3 & 0& 1& 4 \\
363 & 11 & 3 & 0& 1& 5 \\
363 & 11 & 3 & 0& 2& 5 \\
363 & 11 & 3 & 0& 2& 6 \\
384 &  8 & 6 & 0& 1& 3 \\
392 & 14 & 2 & 0& 1& 3 \\
392 & 14 & 2 & 0& 1& 4 \\
\end{tabular}
\quad
\begin{tabular}[t]{cccr@{,}r@{,}r}
$q$& $r$ & $t$ & $d_0$ & $d_1$ & $d_2$\\ \hline
392 & 14 & 2 & 0& 1& 5 \\
392 & 14 & 2 & 0& 1& 6 \\
392 & 14 & 2 & 0& 2& 5 \\
392 & 14 & 2 & 0& 2& 6 \\
392 &  7 & 8$\dag$& 0& 1& 3 \\
392 & 14 & 2 & 0& 3& 8 \\
400 & 20 & 1 & 0& 1& 3 \\
400 & 20 & 1 & 0& 1& 7 \\
400 & 20 & 1 & 0& 2& 6 \\
400 & 10 & 4$\dag$& 0& 1& 3 \\
400 & 20 & 1 & 0& 2& 8 \\
400 & 10 & 4$\dag$& 0& 1& 4 \\
400 & 20 & 1 & 0& 2& 9 \\
400 & 20 & 1 & 0& 3& 9 \\
441 & 21 & 1 & 0& 1& 5 \\
441 & 21 & 1 & 0& 2&10 \\
441 & 21 & 1 & 0& 3& 9 \\
441 &  7 & 9$\dag$& 0& 1& 3 \\
448 &  8 & 7 & 0& 1& 3 \\
484 & 22 & 1 & 0& 1& 3 \\
484 & 22 & 1 & 0& 1& 4 \\
484 & 22 & 1 & 0& 1& 5 \\
484 & 22 & 1 & 0& 1& 6 \\
484 & 22 & 1 & 0& 1& 7 \\
484 & 22 & 1 & 0& 1& 8 \\
484 & 22 & 1 & 0& 1& 9 \\
484 & 22 & 1 & 0& 1&10 \\
484 & 22 & 1 & 0& 2& 5 \\
484 & 22 & 1 & 0& 2& 6 \\
484 & 11 & 4$\dag$& 0& 1& 3 \\
484 & 22 & 1 & 0& 2& 7 \\
484 & 22 & 1 & 0& 2& 8 \\
484 & 11 & 4$\dag$& 0& 1& 4 \\
484 & 22 & 1 & 0& 2& 9 \\
484 & 22 & 1 & 0& 2&10 \\
484 & 11 & 4$\dag$& 0& 1& 5 \\
484 & 22 & 1 & 0& 3& 7 \\
484 & 22 & 1 & 0& 3& 8 \\
484 & 22 & 1 & 0& 3& 9 \\
484 & 22 & 1 & 0& 3&10 \\
484 & 22 & 1 & 0& 3&12 \\
484 & 22 & 1 & 0& 4& 9 \\
484 & 22 & 1 & 0& 4&10 \\
484 & 11 & 4$\dag$& 0& 2& 5 \\
484 & 22 & 1 & 0& 4&12 \\
484 & 11 & 4$\dag$& 0& 2& 6 \\
484 & 22 & 1 & 0& 5&12 \\
484 & 22 & 1 & 0& 5&13 \\
484 & 22 & 1 & 0& 6&13 \\
490 &  7 &10 & 0& 1& 3 \\
500 & 10 & 5 & 0& 1& 3 \\
500 & 10 & 5 & 0& 1& 4 \\
\end{tabular}
\end{table}

\begin{table}
\caption{Pairs of lens spaces $p$-isospectral for all $p$ of dimension $n=7$ and fundamental group of order $q\leq 300$.}\label{table:n=7}
\begin{tabular}[t]{c@{\;\;}c@{\;\;}c@{\;\;}r@{,}r@{,}r@{,}r}
$q$& $r$ & $t$ & $d_0$ & $d_1$ & $d_2$& $d_3$\\ \hline
 49 &  7 & 1 & 0& 1& 2& 4 \\
 81 &  9 & 1 & 0& 1& 2& 4 \\
 81 &  9 & 1 & 0& 1& 2& 5 \\
 81 &  9 & 1 & 0& 1& 3& 5 \\
 98 &  7 & 2 & 0& 1& 2& 4 \\
100 & 10 & 1 & 0& 1& 2& 4 \\
100 & 10 & 1 & 0& 2& 3& 6 \\
121 & 11 & 1 & 0& 1& 2& 4 \\
121 & 11 & 1 & 0& 1& 2& 5 \\
121 & 11 & 1 & 0& 1& 2& 6 \\
121 & 11 & 1 & 0& 1& 3& 5 \\
121 & 11 & 1 & 0& 1& 3& 6 \\
121 & 11 & 1 & 0& 1& 3& 7 \\
121 & 11 & 1 & 0& 1& 4& 6 \\
121 & 11 & 1 & 0& 1& 4& 7 \\
121 & 11 & 1 & 0& 2& 3& 6 \\
121 & 11 & 1 & 0& 2& 4& 7 \\
144 & 12 & 1 & 0& 1& 2& 5 \\
144 & 12 & 1 & 0& 2& 3& 7 \\
147 &  7 & 3 & 0& 1& 2& 4 \\
162 &  9 & 2 & 0& 1& 2& 4 \\
162 &  9 & 2 & 0& 1& 2& 5 \\
162 &  9 & 2 & 0& 1& 3& 5 \\
169 & 13 & 1 & 0& 1& 2& 4 \\
169 & 13 & 1 & 0& 1& 2& 5 \\
169 & 13 & 1 & 0& 1& 2& 6 \\
169 & 13 & 1 & 0& 1& 2& 7 \\
169 & 13 & 1 & 0& 1& 3& 5 \\
169 & 13 & 1 & 0& 1& 3& 6 \\
169 & 13 & 1 & 0& 1& 3& 7 \\
169 & 13 & 1 & 0& 1& 4& 6 \\
169 & 13 & 1 & 0& 1& 4& 7 \\
169 & 13 & 1 & 0& 1& 4& 8 \\
169 & 13 & 1 & 0& 1& 5& 7 \\
169 & 13 & 1 & 0& 1& 5& 8 \\
169 & 13 & 1 & 0& 2& 3& 6 \\
169 & 13 & 1 & 0& 2& 3& 7 \\
169 & 13 & 1 & 0& 2& 3& 8 \\
169 & 13 & 1 & 0& 2& 4& 7 \\
169 & 13 & 1 & 0& 2& 4& 8 \\
169 & 13 & 1 & 0& 2& 5& 9 \\
169 & 13 & 1 & 0& 2& 5& 8 \\
169 & 13 & 1 & 0& 3& 4& 8 \\
196 & 14 & 1 & 0& 1& 2& 4 \\
196 & 14 & 1 & 0& 1& 2& 5 \\
196 & 14 & 1 & 0& 1& 2& 6 \\
196 & 14 & 1 & 0& 1& 3& 5 \\
196 & 14 & 1 & 0& 1& 3& 6 \\
196 & 14 & 1 & 0& 1& 4& 6 \\
196 & 14 & 1 & 0& 1& 4& 9 \\
196 & 14 & 1 & 0& 1& 5& 9 \\
196 & 14 & 1 & 0& 2& 3& 6 \\
\end{tabular}
\;
\begin{tabular}[t]{c@{\;\;}c@{\;\;}c@{\;\;}r@{,}r@{,}r@{,}r}
$q$& $r$ & $t$ & $d_0$ & $d_1$ & $d_2$& $d_3$\\ \hline
196 & 14 & 1 & 0& 2& 3& 8 \\
196 & 14 & 1 & 0& 2& 4& 8 \\
196 &  7 & 4$\dag$& 0& 1& 2& 4 \\
196 & 14 & 1 & 0& 2& 5& 8 \\
196 & 14 & 1 & 0& 3& 4& 8 \\
196 & 14 & 1 & 0& 3& 5& 9 \\
200 & 10 & 2 & 0& 1& 2& 4 \\
200 & 10 & 2 & 0& 2& 3& 6 \\
225 & 15 & 1 & 0& 1& 2& 4 \\
225 & 15 & 1 & 0& 1& 2& 6 \\
225 & 15 & 1 & 0& 1& 2& 8 \\
225 & 15 & 1 & 0& 1& 3& 7 \\
225 & 15 & 1 & 0& 1& 3& 8 \\
225 & 15 & 1 & 0& 1& 4& 8 \\
225 & 15 & 1 & 0& 1& 5& 9 \\
225 & 15 & 1 & 0& 2& 4& 7 \\
225 & 15 & 1 & 0& 2& 4& 8 \\
242 & 11 & 2 & 0& 1& 2& 4 \\
242 & 11 & 2 & 0& 1& 2& 5 \\
242 & 11 & 2 & 0& 1& 2& 6 \\
242 & 11 & 2 & 0& 1& 3& 5 \\
242 & 11 & 2 & 0& 1& 3& 6 \\
242 & 11 & 2 & 0& 1& 3& 7 \\
242 & 11 & 2 & 0& 1& 4& 6 \\
242 & 11 & 2 & 0& 1& 4& 7 \\
242 & 11 & 2 & 0& 2& 3& 6 \\
242 & 11 & 2 & 0& 2& 4& 7 \\
243 &  9 & 3 & 0& 1& 2& 4 \\
243 &  9 & 3 & 0& 1& 2& 5 \\
243 &  9 & 3 & 0& 1& 3& 5 \\
245 &  7 & 5 & 0& 1& 2& 4 \\
256 & 16 & 1 & 0& 1& 2& 5 \\
256 & 16 & 1 & 0& 1& 2& 7 \\
256 & 16 & 1 & 0& 1& 3& 6 \\
256 & 16 & 1 & 0& 1& 4& 7 \\
256 & 16 & 1 & 0& 1& 5&10 \\
256 & 16 & 1 & 0& 2& 3& 7 \\
256 & 16 & 1 & 0& 2& 3& 9 \\
256 & 16 & 1 & 0& 2& 5& 9 \\
256 & 16 & 1 & 0& 3& 4& 9 \\
256 & 16 & 1 & 0& 3& 5&10 \\
288 & 12 & 2 & 0& 1& 2& 5 \\
288 & 12 & 2 & 0& 2& 3& 7 \\
289 & 17 & 1 & 0& 1& 2& 4 \\
289 & 17 & 1 & 0& 1& 2& 5 \\
289 & 17 & 1 & 0& 1& 2& 6 \\
289 & 17 & 1 & 0& 1& 2& 7 \\
289 & 17 & 1 & 0& 1& 2& 8 \\
289 & 17 & 1 & 0& 1& 2& 9 \\
289 & 17 & 1 & 0& 1& 3& 5 \\
289 & 17 & 1 & 0& 1& 3& 6 \\
\end{tabular}
\;
\begin{tabular}[t]{c@{\;\;}c@{\;\;}c@{\;\;}r@{,}r@{,}r@{,}r}
$q$& $r$ & $t$ & $d_0$ & $d_1$ & $d_2$& $d_3$\\ \hline
289 & 17 & 1 & 0& 1& 3& 7 \\
289 & 17 & 1 & 0& 1& 3& 8 \\
289 & 17 & 1 & 0& 1& 3& 9 \\
289 & 17 & 1 & 0& 1& 3&10 \\
289 & 17 & 1 & 0& 1& 4& 6 \\
289 & 17 & 1 & 0& 1& 4& 7 \\
289 & 17 & 1 & 0& 1& 4& 8 \\
289 & 17 & 1 & 0& 1& 4& 9 \\
289 & 17 & 1 & 0& 1& 4&10 \\
289 & 17 & 1 & 0& 1& 5& 7 \\
289 & 17 & 1 & 0& 1& 5& 8 \\
289 & 17 & 1 & 0& 1& 5& 9 \\
289 & 17 & 1 & 0& 1& 5&10 \\
289 & 17 & 1 & 0& 1& 5&11 \\
289 & 17 & 1 & 0& 1& 6& 8 \\
289 & 17 & 1 & 0& 1& 6& 9 \\
289 & 17 & 1 & 0& 1& 6&10 \\
289 & 17 & 1 & 0& 1& 6&11 \\
289 & 17 & 1 & 0& 1& 7& 9 \\
289 & 17 & 1 & 0& 1& 7&10 \\
289 & 17 & 1 & 0& 2& 3& 6 \\
289 & 17 & 1 & 0& 2& 3& 7 \\
289 & 17 & 1 & 0& 2& 3& 8 \\
289 & 17 & 1 & 0& 2& 3& 9 \\
289 & 17 & 1 & 0& 2& 4& 7 \\
289 & 17 & 1 & 0& 2& 4& 8 \\
289 & 17 & 1 & 0& 2& 4& 9 \\
289 & 17 & 1 & 0& 2& 4&10 \\
289 & 17 & 1 & 0& 2& 5& 8 \\
289 & 17 & 1 & 0& 2& 5& 9 \\
289 & 17 & 1 & 0& 2& 5&10 \\
289 & 17 & 1 & 0& 2& 5&11 \\
289 & 17 & 1 & 0& 2& 6& 9 \\
289 & 17 & 1 & 0& 2& 6&10 \\
289 & 17 & 1 & 0& 2& 6&11 \\
289 & 17 & 1 & 0& 2& 7&10 \\
289 & 17 & 1 & 0& 2& 7&11 \\
289 & 17 & 1 & 0& 3& 4& 8 \\
289 & 17 & 1 & 0& 3& 4& 9 \\
289 & 17 & 1 & 0& 3& 4&10 \\
289 & 17 & 1 & 0& 3& 5& 9 \\
289 & 17 & 1 & 0& 3& 5&10 \\
289 & 17 & 1 & 0& 3& 6&10 \\
289 & 17 & 1 & 0& 3& 6&11 \\
289 & 17 & 1 & 0& 3& 7&11 \\
289 & 17 & 1 & 0& 3& 7&12 \\
289 & 17 & 1 & 0& 4& 5&10 \\
289 & 17 & 1 & 0& 4& 6&11 \\
294 &  7 & 6 & 0& 1& 2& 4 \\
300 & 10 & 3 & 0& 1& 2& 4 \\
300 & 10 & 3 & 0& 2& 3& 6 \\
\end{tabular}
\end{table}

\begin{table}
\caption{Pairs of lens spaces $p$-isospectral for all $p$ of dimension $n=9$ and fundamental group of order $q\leq 150$.}\label{table:n=9}
\begin{tabular}[t]{c@{\;\;}c@{\;\;}c@{\;\;}r@{,}r@{,}r@{,}r@{,}r}
$q$& $r$ & $t$ & $d_0$ & $d_1$ & $d_2$& $d_3$& $d_4$\\ \hline
 64 &  8 & 1 & 0& 1& 2& 3& 5 \\
 72 &  ? & ?    \\
 81 &  9 & 1 & 0& 1& 2& 3& 5 \\
 81 &  9 & 1 & 0& 1& 2& 4& 5 \\
 81 &  9 & 1 & 0& 1& 2& 4& 6 \\
121 & 11 & 1 & 0& 1& 2& 3& 5 \\
121 & 11 & 1 & 0& 1& 2& 3& 6 \\
121 & 11 & 1 & 0& 1& 2& 4& 5 \\
121 & 11 & 1 & 0& 1& 2& 4& 6 \\
121 & 11 & 1 & 0& 1& 2& 4& 7 \\
121 & 11 & 1 & 0& 1& 2& 5& 6 \\
121 & 11 & 1 & 0& 1& 2& 5& 7 \\
121 & 11 & 1 & 0& 1& 3& 4& 6 \\
\end{tabular}
\qquad
\begin{tabular}[t]{c@{\;\;}c@{\;\;}c@{\;\;}r@{,}r@{,}r@{,}r@{,}r}
$q$& $r$ & $t$ & $d_0$ & $d_1$ & $d_2$& $d_3$& $d_4$\\ \hline
121 & 11 & 1 & 0& 1& 3& 4& 7 \\
121 & 11 & 1 & 0& 1& 3& 5& 7 \\
121 & 11 & 1 & 0& 1& 3& 5& 8 \\
121 & 11 & 1 & 0& 1& 3& 6& 7 \\
121 & 11 & 1 & 0& 1& 3& 6& 8 \\
121 & 11 & 1 & 0& 1& 4& 5& 7 \\
121 & 11 & 1 & 0& 2& 3& 5& 7 \\
121 & 11 & 1 & 0& 2& 3& 4& 7 \\
128 &  8 & 2 & 0& 1& 2& 3& 5 \\
144 & 12 & 1 & 0& 1& 2& 3& 5 \\
144 & 12 & 1 & 0& 1& 2& 4& 7 \\
144 & 12 & 1 & 0& 2& 3& 5& 7 \\
\end{tabular}
\end{table}

We will show tables with many examples of pairs of lens spaces $p$-isospectral for all $p$ in low dimensions
$n=5$, $7$ and $9$.

The finite-implies-infinite principle mentioned above allowed us to give an algorithm ---implemented in
Sage~\cite{Sage}---
that can find, for each $m$ and $q$,
all pairs of lens spaces of dimension $n=2m-1$ and fundamental group of order $q$
that are $p$-isospectral for all $p$.
This is shown in the tables for $n=5$ and $q\leq 500$,
$n=7$ and $q\leq 300$, and $n=9$ and $q\leq 150$.

On the other hand, Peter Doyle
has implemented a clever computer program using the function $\QQ_L(w,z)$ in Theorem~\ref{thm:allpIkeda} that
can distinguish very quickly whether two lens spaces are $p$-isospectral for all $p$.
We thank Peter for verifying with his method that all of our examples are correct.

For positive integers $r$ and $t$, and $q=r^2t$ (see \cite[\S5]{LMRhodge}) we introduce the element $\ww:=rt+1.$
Clearly $\ww^k \equiv krt+1\pmod{q}$, thus $\ww^r\equiv 1\pmod{q}$.
Since the parameters in most of the lens spaces occurring in the low dimensional examples are congruent to $\pm1\pmod{rt}$, they can be written as powers of $\ww$.
In this way, the basic example can be written as
\begin{equation}\label{eq:r2t-inverses}
\begin{array}{l}
L(49;1,6,15) = L(49;\ww^{0},-\ww^{1}, \ww^{-2}) \cong L(q;\ww^0,\ww^{-1},\ww^{-3}),\\
L(49;1,6,20) = L(49;\ww^{0},-\ww^{1},-\ww^{ 3}) \cong L(q;\ww^0,\ww^{ 1},\ww^{ 3}).
\end{array}
\end{equation}
Here $\cong$ means isometry between lens spaces.
Each entry $q,r,t,(d_0,\dots,d_m)$ in the tables, represents the pair of lens spaces
\begin{equation}\label{eq6:inversos}
L(q;\ww^{d_0},\ww^{d_1},\dots,\ww^{d_{m-1}})
\quad\text{and}\quad
L(q;\ww^{-d_0},\ww^{-d_1},\dots,\ww^{-d_{m-1}}),
\end{equation}
which are $p$-isospectral for all $p$.
When an entry has a dag $\dag$, it means that the pair corresponding to this line is isometric to the pair of the previous line.
This only happens when $t$ is not square-free, thus it is possible to write $q$ as $r^2t$ in more than one way.

Question~5.1 in \cite{LMRhodge} asked for  conditions on $d=(d_0,\dots,d_{m-1})$, $r$ and $t$,  so that the pair of lens spaces in \eqref{eq6:inversos} are $p$-isospectral for every $p$.
Peter Doyle worked on this question in collaboration with D.~DeFord (see \cite{DeFordDoyle14}) and found sufficient conditions by using in a clever way similar techniques as Ikeda.

To state their condition, in their terminology, $d = (d_0,\dots,d_{m-1})$ is said to be:
\begin{itemize}
\item \emph{univalent mod $r$} if its entries are distinct mod $r$,
\item \emph{reversible mod $r$} if the pair of lens spaces in \eqref{eq6:inversos} are isometric,
\item \emph{good mod $r$} if it is univalent or reversible mod $r$,
\item \emph{hereditarily good mod $r$} if it is good mod $c$ for all $c$ dividing $r$.
\end{itemize}
Then, Theorem~1 in \cite{DeFordDoyle14}) says that
\begin{quote}
\emph{if $d$ is hereditarily good mod $r$ and not reversible mod $r$, then the lens spaces in \eqref{eq6:inversos} are $p$-isospectral for all $p$ and are not isometric.}
\end{quote}

By using this condition it is remarkably simple to produce examples.
Indeed, the family in \eqref{eq:infinitefamily} satisfies the conditions of this theorem.
The pairs in \eqref{eq:infinitefamily} has $d=(0,1,3)$ from \eqref{eq:r2t-inverses}.
Now, $d$ is clearly univalent mod $r$ for $r\geq4$, reversible mod $r$ for $r=1,2,4,5$, thus $d$ is good mod $r$ for any $r\neq 3$.
Consequently, $d$ is hereditarily good mod $r$ for any $r$ not divisible by $3$, and not reversible for $r\geq7$.

We believe that this theorem is an important step in the determination of  all the lens spaces that are $p$-isospectral for all $p$.
Though many such examples do come from the  theorem, there are exceptions like the ones we next describe.

\begin{exam}\label{exam:ejemplos-raros}
To each lens space $L(q;\ww^{d_0},\ww^{d_1},\dots,\ww^{d_{m-1}})$ with $0=d_0<d_1<\dots<d_{m-1}<r$ one can associate the ordered partition $$r=(d_1-d_0)+\dots+(d_{m-1}-d_{m-2})+(r-d_{m-1}).$$
One can check that two lens spaces are isometric if their partitions differ by a cyclic reordering.

Not all the known examples of $p$-isospectral lens spaces for all $p$ can be written as in \eqref{eq6:inversos}.
For instance, this is the case for the pair $\overline L,\overline{L'}$, dual to the basic pair \eqref{eq:basicexample}, since their parameters are not necessarily congruent to $\pm1\pmod{rt}$.
Moreover, this phenomenon already occurs in the case of the curious example $L(72;1, 5, 7, 17, 35)$, $L(72;1, 5, 7, 19, 35)$, since neither these lens spaces nor their duals (namely $L(72;1, 5, 7, 11, 19, 25, 35)$ and $L(72;1, 5, 7, 11, 23, 29, 31)$)
 can be written as in~\eqref{eq6:inversos}.
\end{exam}

\begin{problem}\label{problem}
  Determine  all pairs of $n$-dimensional lens spaces $p$-isospectral for all $p$ with fundamental group of order $q$ (or at least all such pairs for infinite values of $n$).
  \end{problem}
  \begin{question}
  Are there families of  non-isometric lens spaces $p$-isospectral for all $p$ having more than two elements?
\end{question}

\subsection{Final remarks}
We end this paper with the following comments.

\begin{rem}
It is shown in \cite[Lemma 7.6]{LMRhodge} that the lens spaces in the family constructed above are  homotopically equivalent to each other.
However, they cannot be simply homotopically equivalent  (see \cite[\S31]{Co}) since in this case they would be isometric.
\end{rem}

\begin{rem}
Despite being $p$-isospectral for every $p$, the $5$-dimensional lens spaces $L$ and $L'$ in our family are `very far' from being strongly isospectral.
Strongly isospectral spherical space forms are necessarily $\tau$-isospectral for every representation $\tau$ of $\SO(5)$.
However, in the case at hand, one can explicitly show many representations $\tau$ of $\SO(5)$ such that lens spaces $ L, L'$ are not $\tau$-isospectral.

If we look at the basic case $L=L(49;1,6,15)$ and $L' = L(49;1,6,20)$, in \cite[\S8]{LMRhodge} we show that if  $\pi_0$ is the unitary irreducible representation of $\SO(6)$ with highest weight $\Lambda_0=4\varepsilon_1+3\varepsilon_2$, then the lens spaces $L$ and $L'$ are not $\tau$-isospectral for every irreducible representation $\tau$ of $\SO(5)$ with highest weight of the form $b_1\varepsilon_1 + b_2\varepsilon_2$  for $4\geq b_1\geq 3\geq b_2\geq0$.

By computer methods, by using Sage~\cite{Sage}, we checked that there are  many   choices  $\pi_0$ with this property,
thus providing many other $K$-types $\tau$ such that the lens spaces $L$ and $L'$ are not $\tau$-isospectral.
In this connection, in \cite{LMRhodge} we make the following conjecture:
\begin{quote}
\emph{There are only finitely many irreducible representations $\tau$ of $K=\SO(5)$ such that $L$ and $L'$ are $\tau$-isospectral}.
\end{quote}
\end{rem}

\begin{rem}\label{rem:Dirac-case}
Recently, Sebastian Boldt and the first named author in \cite{Boldt-Lauret-one-norm-Dirac} extended the methods in this paper to the Dirac operator on lens spaces admitting a spin structure.
As it can be expected, the Dirac case involves more technical difficulties.
In this case, one associates to a lens space $L=L(q;s_1,\dots,s_m)$ with a fixed spin structure, an \emph{affine congruence lattice} $\mathcal L$.
When $q$ and $m$ are odd, $L$ admits exactly one spin structure and the associated $\mathcal L$ is given by
\begin{equation}\label{eq:mathcalLDirac}
   \mathcal L =\{(a_1,\dots,a_m) \in (\tfrac12+\Z)^m : 2(a_1s_1+\dots +a_ms_m)\equiv0\pmod q\}.
\end{equation}
The case $q$ even is a bit more involved since $L$ admits two spin structures.

In analogy with Theorem~\ref{thm:0-characterization}, the authors show that two lens spaces are Dirac isospectral if and only if their associated affine congruence lattices are $\norma{\cdot}$-isospectral. 
Furthermore, they are able to construct the following examples:
\begin{itemize}
\item an increasing family of lens spaces mutually Dirac isospectral with increasing dimension;
\item an infinite sequence of $7$-dimensional lens spaces, each of them with two Dirac isospectral spin structures;
\item an infinite sequence of pairs of non-isometric $7$-dimensional lens spaces admitting exactly one spin structure that are Dirac isospectral.
\end{itemize}
\end{rem}

\begin{rem}
 All  examples in the literature of pairs of isospectral spherical space forms with non-cyclic fundamental group (i.e.\ lens spaces are not allowed) are obtained by Sunada's method, hence they are strongly isospectral and correspond to almost conjugate subgroups of $\SO(2m)$   (\cite{Ik83}, \cite{Gi}, \cite{Wo2}).

\begin{question}
Can one construct $0$-isospectral spherical space forms with non-cyclic fundamental groups that are not strongly isospectral?
\end{question}

Note that two $3$-dimensional isospectral spherical space forms are isometric (\cite{IY}, \cite{Y}, \cite{Ik80a}), so the answer is negative in dimension $3$.
Furthermore, Wolf in \cite[Cor.~7.3]{Wo2} showed the non-existence of such examples for any dimension $2m-1$ with $m$ prime (see also \cite[Thm.~3.1 and Thm.~3.9]{Ik80c}).
\end{rem}

\bibliographystyle{plain}

\end{document}